\documentclass[a4paper, 11pt]{article}
\usepackage[top=30truemm, bottom=30truemm, left=25truemm, right=25truemm]{geometry}
\usepackage{amsmath, amsthm, amssymb}

\theoremstyle{definition}
\newtheorem{definition}{Definition}[section]
\newtheorem{example}[definition]{Example}

\theoremstyle{plain}
\newtheorem{lemma}[definition]{Lemma}
\newtheorem{proposition}[definition]{Proposition}
\newtheorem{theorem}[definition]{Theorem}
\newtheorem{corollary}[definition]{Corollary}

\title{Bivariate Affine $q$-Krawtchouk Polynomials and Association Schemes over Galois Rings}
\author{Yuta Watanabe}
\date{July 2026}

\begin{document}

\maketitle

\begin{center}
Dedicated to Paul Terwilliger on his 70th birthday
\end{center}

\begin{abstract}
We introduce regularized bivariate affine $q$-Krawtchouk polynomials and show that they give the first eigenmatrix of a translation association scheme on $\operatorname{Mat}_{d\times n}(\operatorname{GR}(p^2,r))$.
The relations are defined by Smith type, and the corresponding eigenvalues are written as character sums over Smith type classes.
The proof is based on recurrence relations obtained from the transition numbers describing how Smith types change when one row and one column are added to a matrix.
These recurrences identify the character sums with the regularized bivariate affine $q$-Krawtchouk polynomials.
\end{abstract}

\section{Introduction}

The first eigenmatrix, or character table, of the Hamming association scheme is classically expressed in terms of the Krawtchouk polynomials \cite{BBIT,D73}.
Several generalizations of these polynomials have been studied.
One direction is the multivariate extension, which arises in ordered Hamming schemes and constructions related to character algebras \cite{MT}.
Another is the $q$-analog; for the bilinear forms scheme, the first eigenmatrix is described by the affine $q$-Krawtchouk polynomials \cite{D78}.

In this paper we study a translation association scheme on
$X=\operatorname{Mat}_{d\times n}(\operatorname{GR}(p^2,r))$,
where the relations are defined by Smith type.
In this setting, the first eigenmatrix is expressed in terms of regularized bivariate affine $q$-Krawtchouk polynomials.
Kawamura~\cite{K} studied multivariate Krawtchouk and Hahn polynomials from the viewpoint of harmonic analysis on a non-Archimedean local field and representations of wreath products.
Guo and Li~\cite{GL} studied Smith normal forms, orbit decompositions, and orbit lengths for matrices over Galois rings, with applications including association schemes.
Their construction of an association scheme is based on $1\times n$ matrices over a Galois ring.
The present paper instead treats $d\times n$ matrices and determines the first eigenmatrix of the resulting translation association scheme.

Let $q=p^r$ and $R=\operatorname{GR}(p^2,r)$.
We regard $X=\operatorname{Mat}_{d\times n}(R)$ as an additive group.
A matrix over $R$ has Smith type $(i,j)$ if its Smith normal form has $i$ invariant factors equal to $p$, $j$ invariant factors equal to $1$, and all remaining invariant factors equal to $0$.
The relations of the scheme are defined by the Smith type of the difference of two matrices.
We introduce regularized bivariate affine $q$-Krawtchouk polynomials and prove that their values give the entries of the first eigenmatrix of this scheme.

A technical point is the specialization of the polynomial parameter $a$.
The value needed for the association scheme is $a=q^d$,
while the usual hypergeometric expression for the affine $q$-Krawtchouk polynomials may have apparent singularities of this form.
We therefore use a regularized form of the affine $q$-Krawtchouk polynomials.
The bivariate polynomials in this paper are defined in terms of this regularized form, so that the specialization $a=q^d$ is well-defined.

The proof of the main theorem is based on recurrence relations.
The regularized bivariate affine $q$-Krawtchouk polynomials satisfy two recurrence relations derived from those for the regularized one-variable affine $q$-Krawtchouk polynomials.
The eigenvalues of the association scheme are written as character sums over Smith type classes.
The main enumeration determines the transition numbers describing how Smith types change when one row and one column are added to a matrix.
These transition numbers are then used to show that the character sums satisfy the same recurrence relations as the regularized bivariate affine $q$-Krawtchouk polynomials.

This paper is organized as follows.
Section~2 reviews affine $q$-Krawtchouk polynomials and introduces the regularized form needed for later specialization.
Section~3 defines the regularized bivariate affine $q$-Krawtchouk polynomials and establishes their recurrence
relations.
Section~4 introduces the association scheme on matrices over Galois rings, states the main theorem, and includes a worked example for $\operatorname{Mat}_{2\times2}(\mathbb{Z}_4)$.
Section~5 carries out the key enumeration of matrices needed for the proof.
Section~6 derives the recurrence relations for the eigenvalue sums and proves the main theorem.
The finite-field counting lemmas used in Section~5 are collected in the appendix.

\section{Affine $q$-Krawtchouk polynomials}

Let $\mathbb K$ be a field, and let
$q\in\mathbb K^\times$ be not a root of unity. 
We use the standard notation for $q$-shifted factorials
\[
(x;q)_0=1,\qquad
(x;q)_h=\prod_{\nu=0}^{h-1}(1-xq^\nu)
\quad (h\geq 1),
\]
and the basic hypergeometric series notation
\[
{}_3\phi_2
\left(
\begin{matrix}
x_1,x_2,x_3\\
y_1,y_2
\end{matrix}
;q,z
\right)
=
\sum_{h=0}^\infty
\frac{(x_1;q)_h(x_2;q)_h(x_3;q)_h}
     {(y_1;q)_h(y_2;q)_h(q;q)_h}
z^h.
\]
For $0\leq m\leq n$, the $q$-binomial coefficient is defined by
\[
\begin{bmatrix} n\\ m \end{bmatrix}_q
=
\frac{(q;q)_n}{(q;q)_m(q;q)_{n-m}},
\]
and we set
\[
\begin{bmatrix} n\\ m \end{bmatrix}_q=0
\quad\text{if } m<0 \text{ or } m>n.
\]
We shall use the following elementary identity for the $q$-binomial coefficients:
\begin{equation}
\label{eq:qbinom-switch}
\begin{bmatrix} n\\ m_1 \end{bmatrix}_q
\begin{bmatrix} n-m_1\\ m_2 \end{bmatrix}_q
=
\begin{bmatrix} n\\ m_2 \end{bmatrix}_q
\begin{bmatrix} n-m_2\\ m_1 \end{bmatrix}_q,
\qquad
(m_1,m_2 \geq 0,\quad m_1+m_2 \le n).
\end{equation}
This follows immediately from the definition, since both sides are equal to
\[
\frac{(q;q)_n}
     {(q;q)_{m_1}(q;q)_{m_2}(q;q)_{n-m_1-m_2}}.
\]
Under the above assumption on $q$,
the $q$-shifted factorials in the denominators that depend only on $q$,
such as $(q;q)_h$ and $(q^{-n};q)_h$,
are nonzero in the ranges used below.

We recall the definition and basic properties of the affine $q$-Krawtchouk polynomials from \cite[Section~14.16]{KLS}.
For $n \in \mathbb{Z}_{\ge 0}$ and for a generic value of $a\in\mathbb K^\times$, namely when
\begin{equation}
\label{eq:a-generic-assumption}
a\ne q^\alpha
\qquad
(0\le \alpha < n),
\end{equation}
the normalized affine $q$-Krawtchouk polynomials of size $n$ are given by
\[
K_i^{\mathrm{aff}}(j;a,n;q)
=
W_i(a;n,q)\,
{}_3\phi_2
\left(
\begin{matrix}
q^{-i},0,q^{-j}\\
q^{-n},a^{-1}
\end{matrix}
;q,q
\right),
\qquad
(0\leq i,j\leq n),
\]
where
\[
W_i(a;n,q)
=
\begin{bmatrix} n\\ i \end{bmatrix}_q
\prod_{\nu=0}^{i-1}(a-q^\nu).
\]
Here and in what follows, the empty product is understood to be $1$.
These are normalized versions of the usual affine $q$-Krawtchouk
polynomials.

The genericity assumption \eqref{eq:a-generic-assumption}
is needed only for the hypergeometric expression above. In the
normalized expression, the apparent singularities at the values
$a=q^\alpha$ are removable. Indeed, for $0\le h\le i$, we have
\[
\frac{\prod_{\nu=0}^{i-1}(a-q^\nu)}{(a^{-1};q)_h}
=
a^h\prod_{\nu=h}^{i-1}(a-q^\nu).
\]
Thus each summand in the normalized expression can be rewritten without
the factor $(a^{-1};q)_h$ in the denominator.

Motivated by this cancellation, we define the regularized affine
$q$-Krawtchouk polynomial by
\begin{equation}
\label{eq:regularized-formula}
\widetilde K_i^{\mathrm{aff}}(j;a,n;q)
=
\begin{bmatrix} n\\ i \end{bmatrix}_q
\sum_{h=0}^{\min\{i,j\}}
\frac{(q^{-i};q)_h(q^{-j};q)_h}
     {(q^{-n};q)_h(q;q)_h}
a^h q^h \prod_{\nu=h}^{i-1}(a-q^\nu).
\end{equation}
The right-hand side of \eqref{eq:regularized-formula} is a polynomial in $a$,
and hence it is well-defined at the exceptional values $a=q^\alpha$.

\begin{lemma}
For $n \in \mathbb{Z}_{\ge 0}$ and $a\in\mathbb K^\times$ satisfying the genericity assumption \eqref{eq:a-generic-assumption}, we have
\[
\widetilde K_i^{\mathrm{aff}}(j;a,n;q)
=
K_i^{\mathrm{aff}}(j;a,n;q),
\qquad
(0 \le i,j \le n).
\]
In particular, $\widetilde K_i^{\mathrm{aff}}(j;a,n;q)$ gives a
regularized extension of the normalized affine $q$-Krawtchouk
polynomial to the values of $a$ at which the hypergeometric
expression may be singular.
\end{lemma}

\begin{proof}
Expanding the basic hypergeometric series and using $(0;q)_h=1$, we
obtain
\[
K_i^{\mathrm{aff}}(j;a,n;q)
=
\begin{bmatrix} n\\ i \end{bmatrix}_q
\sum_{h=0}^\infty
\frac{(q^{-i};q)_h(q^{-j};q)_h}
     {(q^{-n};q)_h(q;q)_h}
q^h
\frac{\prod_{\nu=0}^{i-1}(a-q^\nu)}
     {(a^{-1};q)_h}.
\]
Since $(q^{-i};q)_h=0$ for $h>i$ and
$(q^{-j};q)_h=0$ for $h>j$, the sum terminates at
$h=\min\{i,j\}$.
For $0\leq h\leq i$,
\[
\frac{\prod_{\nu=0}^{i-1}(a-q^\nu)}
     {(a^{-1};q)_h}
=
a^h \prod_{\nu=h}^{i-1}(a-q^\nu).
\]
Substituting this identity into the preceding expansion yields precisely
$\widetilde K_i^{\mathrm{aff}}(j;a,n;q)$.
\end{proof}

In what follows, we write
\[
K_i^{\mathrm{aff}}(j;a,n;q)
\]
for the regularized polynomial
$\widetilde K_i^{\mathrm{aff}}(j;a,n;q)$.
Thus the standing assumptions that $q\in\mathbb K^\times$ is not a
root of unity and $a\in\mathbb K^\times$ remain in force, whereas the
genericity assumption \eqref{eq:a-generic-assumption}, which was
required only for the hypergeometric expression, is not imposed unless
explicitly stated. In particular, the specializations $a=q^m$,
including $a=1$, are allowed.
For convenience, we set
\[
K_i^{\mathrm{aff}}(j;a,n;q) = 0
\qquad \text{if $i < 0$ or $i > n$.}
\]

For the purpose of extending the identities below, we regard $a$ as an indeterminate.
The identities for affine $q$-Krawtchouk polynomials used below are known for generic values of the parameter $a$.
After replacing $K_i^{\mathrm{aff}}$ by the regularized polynomial,
the relevant identities become polynomial identities in $a$, after clearing denominators if necessary.
Hence they may be specialized to any $a\in\mathbb K^\times$.

For $n \in \mathbb{Z}_{\ge 0}$ and $a\in\mathbb K^\times$,
the following initial conditions for the regularized affine $q$-Krawtchouk polynomials hold without imposing the genericity assumption \eqref{eq:a-generic-assumption}:
\begin{equation}
\label{aff-initial}
K_0^{\mathrm{aff}}(j;a,n;q)=1
\quad (0\leq j\leq n),
\qquad
K_i^{\mathrm{aff}}(0;a,n;q)=W_i(a;n,q)
\quad (0\leq i\leq n).
\end{equation}

For $n\in\mathbb Z_{\ge0}$ and for generic values of $a$,
namely for $a\in\mathbb K^\times$ satisfying \eqref{eq:a-generic-assumption},
these polynomials also satisfy the orthogonality relation
\begin{equation}
\label{aff-orthogonal}
\sum_{j=0}^{n}
K_i^{\mathrm{aff}}(j;a,n;q)
K_{i'}^{\mathrm{aff}}(j;a,n;q)
W_j(a;n,q)
=
\delta_{i,i'}\,a^n W_i(a;n,q),
\qquad
(0 \leq i,i' \leq n).
\end{equation}
Since both sides of \eqref{aff-orthogonal} are polynomials in $a$ and agree for generic $a$, the identity holds for all $a\in\mathbb K^\times$.
At specializations excluded from the generic hypergeometric definition, some weights may vanish.
Therefore, we do not use the non-degeneracy of this orthogonality relation in the sequel.

For $n\in\mathbb Z_{\ge0}$ and $a\in\mathbb K^\times$, the
regularized affine $q$-Krawtchouk polynomials satisfy the following
recurrence relations. These identities are first obtained for generic
$a$ from the usual affine $q$-Krawtchouk polynomials and then
extended to all $a\in\mathbb K^\times$ by regularization.

For $0\leq i\leq n$ and $0\leq j\leq n$, the three-term recurrence is
\begin{equation}
\label{aff-three-term}
\begin{aligned}
a q^n(q^{-j}-1)\,K_i^{\mathrm{aff}}(j;a,n;q)
&= q^i(q^{i+1}-1)\,
   K_{i+1}^{\mathrm{aff}}(j;a,n;q) \\
&\quad
-\Bigl\{(q^n-q^i)(a-q^i)
      +q^{i-1}(q^i-1)\Bigr\}
  K_i^{\mathrm{aff}}(j;a,n;q) \\
&\quad
+(q^n-q^{i-1})(a-q^{i-1})\,
 K_{i-1}^{\mathrm{aff}}(j;a,n;q).
\end{aligned}
\end{equation}

For $0\leq i\leq n+1$ and $0\leq j\leq n$, the forward shift relation is
\begin{equation}
\label{aff-forward}
K_i^{\mathrm{aff}}(j+1;aq,n+1;q)
-
K_i^{\mathrm{aff}}(j;aq,n+1;q)
=
-aq^{n-j+1}
K_{i-1}^{\mathrm{aff}}(j;a,n;q).
\end{equation}

For $0\leq i\leq n+1$ and $0\leq j\leq n$, the backward shift relation is
\begin{equation}
\label{aff-backward}
q^iK_i^{\mathrm{aff}}(j;a,n;q)
-
q^{i-1}K_{i-1}^{\mathrm{aff}}(j;a,n;q)
=
K_i^{\mathrm{aff}}(j+1;aq,n+1;q).
\end{equation}
For $0\leq i\leq n$, this is the usual backward shift relation. The
endpoint $i=n+1$ follows from the same identity together with the
convention $K_i^{\mathrm{aff}}(j;a,n;q)=0$ for $i>n$.
Alternatively, it follows by direct substitution into the regularized expression.

\begin{lemma}
\label{lem:aff-reduction}
Let $n \in \mathbb{Z}_{\ge 0}$ and $a\in\mathbb K^\times$.
For $0\leq i\leq n+1$ and $0\leq j\leq n$, the following relation holds:
\begin{equation*}
\begin{aligned}
K_i^{\mathrm{aff}}(j;aq,n+1;q)
&= q^{2i}\,K_i^{\mathrm{aff}}(j;a,n;q) \\
&\quad
+q^i\bigl(q^n+a-q^{i-1}-q^{i-2}\bigr)\,
 K_{i-1}^{\mathrm{aff}}(j;a,n;q) \\
&\quad
+q\bigl(q^n-q^{i-2}\bigr)\bigl(a-q^{i-2}\bigr)\,
 K_{i-2}^{\mathrm{aff}}(j;a,n;q).
\end{aligned}
\end{equation*}
\end{lemma}

\begin{proof}
The case $i=0$ follows from the initial condition \eqref{aff-initial} and the convention
$K_i^{\mathrm{aff}}(j;a,n;q)=0$ for $i<0$. Hence, we may assume
$i\geq 1$.

By the forward shift relation \eqref{aff-forward}, we have
\[
K_i^{\mathrm{aff}}(j;aq,n+1;q)
=
aq^{n-j+1}K_{i-1}^{\mathrm{aff}}(j;a,n;q)
+
K_i^{\mathrm{aff}}(j+1;aq,n+1;q).
\]
Applying the backward shift relation \eqref{aff-backward} to the second term on the
right-hand side gives
\[
K_i^{\mathrm{aff}}(j;aq,n+1;q)
=
aq^{n-j+1}K_{i-1}^{\mathrm{aff}}(j;a,n;q)
+
q^iK_i^{\mathrm{aff}}(j;a,n;q)
-
q^{i-1}K_{i-1}^{\mathrm{aff}}(j;a,n;q).
\]
Thus
\begin{equation*}
\begin{aligned}
K_i^{\mathrm{aff}}(j;aq,n+1;q)
&=
aq^{n+1}(q^{-j}-1)
K_{i-1}^{\mathrm{aff}}(j;a,n;q) \\
&\quad
+\bigl(aq^{n+1}-q^{i-1}\bigr)
K_{i-1}^{\mathrm{aff}}(j;a,n;q) \\
&\quad
+q^iK_i^{\mathrm{aff}}(j;a,n;q).
\end{aligned}
\end{equation*}

We now apply the three-term recurrence \eqref{aff-three-term} with $i$ replaced by $i-1$,
and then multiply the resulting identity by $q$. This gives
\begin{equation*}
\begin{aligned}
&aq^{n+1}(q^{-j}-1)
K_{i-1}^{\mathrm{aff}}(j;a,n;q) \\
&\quad =
q^i(q^i-1)\,
K_i^{\mathrm{aff}}(j;a,n;q) \\
&\qquad
-q\Bigl\{(q^n-q^{i-1})(a-q^{i-1})
      +q^{i-2}(q^{i-1}-1)\Bigr\}
K_{i-1}^{\mathrm{aff}}(j;a,n;q) \\
&\qquad
+q(q^n-q^{i-2})(a-q^{i-2})\,
K_{i-2}^{\mathrm{aff}}(j;a,n;q).
\end{aligned}
\end{equation*}

Substituting this into the previous expression, the coefficient of
$K_i^{\mathrm{aff}}(j;a,n;q)$ becomes
\[
q^i(q^i-1)+q^i=q^{2i}.
\]
The coefficient of $K_{i-1}^{\mathrm{aff}}(j;a,n;q)$ is
\[
\begin{aligned}
&-q\Bigl\{(q^n-q^{i-1})(a-q^{i-1})
      +q^{i-2}(q^{i-1}-1)\Bigr\}
+aq^{n+1}-q^{i-1} \\
&\qquad
=
q^i\bigl(q^n+a-q^{i-1}-q^{i-2}\bigr).
\end{aligned}
\]
Therefore,
\[
\begin{aligned}
K_i^{\mathrm{aff}}(j;aq,n+1;q)
&= q^{2i}\,K_i^{\mathrm{aff}}(j;a,n;q) \\
&\quad
+q^i\bigl(q^n+a-q^{i-1}-q^{i-2}\bigr)\,
 K_{i-1}^{\mathrm{aff}}(j;a,n;q) \\
&\quad
+q\bigl(q^n-q^{i-2}\bigr)\bigl(a-q^{i-2}\bigr)\,
 K_{i-2}^{\mathrm{aff}}(j;a,n;q).
\end{aligned}
\]
\end{proof}

We record for later use a simple consequence of the regularized formula at the specializations $a=q^d$.

\begin{lemma}
\label{lem:aff-boundary-vanishing}
Let $d,n \in \mathbb Z_{\ge0}$ and let $0\le j \le d < i \le n$.
Then
\[
K^{\mathrm{aff}}_i(j;q^d,n;q)=0.
\]
\end{lemma}
\begin{proof}
In each summand of the regularized formula \eqref{eq:regularized-formula}, the summation index $h$ satisfies $h\le j\le d$.
Since $d<i$, the product
\[
\prod_{\nu=h}^{i-1}(q^d-q^\nu)
\]
contains the factor corresponding to $\nu=d$, namely $q^d-q^d=0$.
Hence every summand is zero.
\end{proof}

\section{Bivariate affine $q$-Krawtchouk polynomials}

This section uses the regularized affine $q$-Krawtchouk polynomials
introduced in Section~2. All identities below are understood in this
regularized sense.
In particular, identities obtained initially for generic $a$ will be
used after being extended as polynomial identities in $a$, so that
specializations such as $a=q^m$, including $a=1$, are allowed.

Fix $n \in \mathbb{Z}_{\ge 0}$ and $a\in\mathbb K^\times$ and let
\[
I_n
=
\{(i_1,i_2)\in \mathbb{Z}_{\geq 0}^2
\mid i_1+i_2\leq n\}.
\]

\begin{definition}
\label{def:biaff-qK}
For $(i_1,i_2),(j_1,j_2)\in I_n$, the bivariate affine
$q$-Krawtchouk polynomial is defined by
\[
K_{i_1,i_2}^{\mathrm{biaff}}(j_1,j_2;a,n;q)
=
a^{i_2}q^{i_2(n-i_2)}
K_{i_1}^{\mathrm{aff}}(j_2;aq^{-i_2},n-i_2;q)
K_{i_2}^{\mathrm{aff}}(j_1;aq^{-j_2},n-j_2;q).
\]
Here $K_i^{\mathrm{aff}}$ denotes the regularized affine
$q$-Krawtchouk polynomial from Section~2.
\end{definition}

By the convention for the affine $q$-Krawtchouk polynomials, if
$i_2+j_2>n$, then the second affine factor has degree index $i_2$
larger than its size $n-j_2$. Hence
\[
K_{i_1,i_2}^{\mathrm{biaff}}(j_1,j_2;a,n;q)=0
\qquad (i_2+j_2>n).
\]
We also use the convention that
\[
K_{i_1,i_2}^{\mathrm{biaff}}(j_1,j_2;a,n;q)=0
\]
whenever $(i_1,i_2)\notin I_n$.

With $a$ fixed, the limit as $q\to1$ of the polynomials in Definition~\ref{def:biaff-qK} gives, up to normalization and a parameter identification, a bivariate Krawtchouk system of Griffiths~\cite{G}.

For $(i_1,i_2)\in I_n$, define
\[
W_{i_1,i_2}(a;n,q)
=
\begin{bmatrix} n\\ i_2 \end{bmatrix}_q
\begin{bmatrix} n-i_2\\ i_1 \end{bmatrix}_q
a^{i_2}q^{i_2(n-i_1-i_2)}
\prod_{\nu=0}^{i_1+i_2-1}(a-q^\nu).
\]
The weight admits the factorization
\begin{equation}
\label{eq:W-factorization}
W_{i_1,i_2}(a;n,q)
=
a^{i_2}q^{i_2(n-i_2)}
W_{i_2}(a;n,q)
W_{i_1}(aq^{-i_2};n-i_2,q).
\end{equation}
Indeed,
\[
W_{i_1}(aq^{-i_2};n-i_2,q)
=
\begin{bmatrix} n-i_2\\ i_1 \end{bmatrix}_q
q^{-i_1i_2}
\prod_{\nu=i_2}^{i_1+i_2-1}(a-q^\nu),
\]
and hence the displayed factorization follows directly from the
definition of $W_{i_1,i_2}(a;n,q)$.

\begin{proposition}
Assume the genericity assumption \eqref{eq:a-generic-assumption}.
For $(i_1,i_2), (i'_1,i'_2) \in I_n$, the bivariate affine $q$-Krawtchouk polynomials
satisfy
\[
\sum_{(j_1,j_2)\in I_n}
K_{i_1,i_2}^{\mathrm{biaff}}(j_1,j_2;a,n;q)
K_{i'_1,i'_2}^{\mathrm{biaff}}(j_1,j_2;a,n;q)
W_{j_1,j_2}(a;n,q)
=
\delta_{i_1,i'_1}\delta_{i_2,i'_2}
a^{2n}W_{i_1,i_2}(a;n,q).
\]
After regularization, the displayed identity extends as a polynomial
identity in $a$. At specializations excluded from the generic
hypergeometric definition, some weights may vanish; hence we do not
use the non-degeneracy of this orthogonality relation in the sequel.
\end{proposition}

\begin{proof}
For generic $a$, the proof is obtained by applying the affine
orthogonality relation \eqref{aff-orthogonal} twice.

By Definition~\ref{def:biaff-qK}, the left-hand side of the desired
orthogonality relation is
\[
\begin{aligned}
&\sum_{(j_1,j_2)\in I_n}
K_{i_1,i_2}^{\mathrm{biaff}}(j_1,j_2;a,n;q)
K_{i'_1,i'_2}^{\mathrm{biaff}}(j_1,j_2;a,n;q)
W_{j_1,j_2}(a;n,q)\\
&=
\sum_{(j_1,j_2)\in I_n}
a^{i_2+i'_2}
q^{i_2(n-i_2)+i'_2(n-i'_2)}
K_{i_1}^{\mathrm{aff}}(j_2;aq^{-i_2},n-i_2;q)
K_{i'_1}^{\mathrm{aff}}(j_2;aq^{-i'_2},n-i'_2;q)\\
&\qquad\qquad\qquad\qquad
\times
K_{i_2}^{\mathrm{aff}}(j_1;aq^{-j_2},n-j_2;q)
K_{i'_2}^{\mathrm{aff}}(j_1;aq^{-j_2},n-j_2;q)
W_{j_1,j_2}(a;n,q).
\end{aligned}
\]
Using the factorization \eqref{eq:W-factorization} with $(i_1,i_2)$ replaced by $(j_1,j_2)$, we first fix
$j_2$ and sum over $j_1$. Applying \eqref{aff-orthogonal} with
parameter $aq^{-j_2}$ and size $n-j_2$, we obtain
\[
\begin{aligned}
&\sum_{j_1=0}^{n-j_2}
K_{i_2}^{\mathrm{aff}}(j_1;aq^{-j_2},n-j_2;q)
K_{i'_2}^{\mathrm{aff}}(j_1;aq^{-j_2},n-j_2;q)
W_{j_1}(aq^{-j_2};n-j_2,q)\\
&\qquad
=
\delta_{i_2,i'_2}
(aq^{-j_2})^{n-j_2}
W_{i_2}(aq^{-j_2};n-j_2,q).
\end{aligned}
\]
Therefore the original sum becomes
\[
\begin{aligned}
&\delta_{i_2,i'_2}
a^{n+2i_2}q^{2i_2(n-i_2)}
\sum_{j_2=0}^{n-i_2}
K_{i_1}^{\mathrm{aff}}(j_2;aq^{-i_2},n-i_2;q)
K_{i'_1}^{\mathrm{aff}}(j_2;aq^{-i_2},n-i_2;q)\\
&\qquad\qquad\qquad\qquad
\times
W_{j_2}(a;n,q)
W_{i_2}(aq^{-j_2};n-j_2,q).
\end{aligned}
\]

We next rewrite the remaining product of weights.
By the $q$-binomial identity
\eqref{eq:qbinom-switch}, together with the product definitions of the
one-variable weights, we have
\[
W_{j_2}(a;n,q)W_{i_2}(aq^{-j_2};n-j_2,q)
=
W_{i_2}(a;n,q)W_{j_2}(aq^{-i_2};n-i_2,q).
\]
Indeed, the equality of the $q$-binomial factors is exactly
\eqref{eq:qbinom-switch}, and the remaining product on each side is
\[
q^{-i_2j_2}\prod_{\nu=0}^{i_2+j_2-1}(a-q^\nu).
\]
Thus the sum becomes
\[
\begin{aligned}
&\delta_{i_2,i'_2}
a^{n+2i_2}q^{2i_2(n-i_2)}
W_{i_2}(a;n,q)\\
&\quad\times
\sum_{j_2=0}^{n-i_2}
K_{i_1}^{\mathrm{aff}}(j_2;aq^{-i_2},n-i_2;q)
K_{i'_1}^{\mathrm{aff}}(j_2;aq^{-i_2},n-i_2;q)
W_{j_2}(aq^{-i_2};n-i_2,q).
\end{aligned}
\]
Applying the affine orthogonality relation \eqref{aff-orthogonal} a
second time, now with parameter $aq^{-i_2}$ and size $n-i_2$, gives
\[
\delta_{i_1,i'_1}
(aq^{-i_2})^{n-i_2}
W_{i_1}(aq^{-i_2};n-i_2,q).
\]
Hence the whole expression is
\[
\delta_{i_1,i'_1}\delta_{i_2,i'_2}
a^{2n+i_2}q^{i_2(n-i_2)}
W_{i_2}(a;n,q)
W_{i_1}(aq^{-i_2};n-i_2,q).
\]
Finally, by the factorization \eqref{eq:W-factorization},
\[
a^{i_2}q^{i_2(n-i_2)}
W_{i_2}(a;n,q)
W_{i_1}(aq^{-i_2};n-i_2,q)
=
W_{i_1,i_2}(a;n,q).
\]
Therefore the expression is
\[
\delta_{i_1,i'_1}\delta_{i_2,i'_2}
a^{2n}W_{i_1,i_2}(a;n,q),
\]
for generic $a$. Since both sides are polynomials in $a$, the identity extends to the regularized polynomials by the argument in Section~2.
\end{proof}

\begin{lemma}
\label{lem:biaff-initial}
The bivariate affine $q$-Krawtchouk polynomials satisfy
\[
K_{0,0}^{\mathrm{biaff}}(j_1,j_2;a,n;q)=1
\quad ((j_1,j_2)\in I_n),
\]
and
\[
K_{i_1,i_2}^{\mathrm{biaff}}(0,0;a,n;q)
=
W_{i_1,i_2}(a;n,q)
\quad ((i_1,i_2)\in I_n).
\]
\end{lemma}

\begin{proof}
This follows directly from the initial conditions \eqref{aff-initial} for
$K_i^{\mathrm{aff}}$, the factorization \eqref{eq:W-factorization}, and Definition~\ref{def:biaff-qK}.
\end{proof}

\begin{lemma}
\label{lem:biaff_RR2}
Assume that $j_2\geq 1$. For
$(i_1,i_2),(j_1,j_2)\in I_n$, the following relation holds:
\[
\begin{aligned}
K_{i_1,i_2}^{\mathrm{biaff}}(j_1,j_2;a,n;q)
&=
q^{i_1+2i_2}
K_{i_1,i_2}^{\mathrm{biaff}}(j_1,j_2-1;aq^{-1},n-1;q)\\
&\quad
-
q^{i_1+2i_2-1}
K_{i_1-1,i_2}^{\mathrm{biaff}}(j_1,j_2-1;aq^{-1},n-1;q).
\end{aligned}
\]
\end{lemma}

\begin{proof}
If $i_2+j_2>n$, then the left-hand side is zero by the convention
following Definition~\ref{def:biaff-qK}. Moreover,
\[
i_2+(j_2-1)>n-1,
\]
so both bivariate affine $q$-Krawtchouk polynomials on the
right-hand side are also zero by the same convention, now with size
$n-1$. Hence the relation is trivial in this case. We may therefore
assume that $i_2+j_2\leq n$.

By Definition~\ref{def:biaff-qK},
\[
K_{i_1,i_2}^{\mathrm{biaff}}(j_1,j_2;a,n;q)
=
a^{i_2}q^{i_2(n-i_2)}
K_{i_1}^{\mathrm{aff}}(j_2;aq^{-i_2},n-i_2;q)
K_{i_2}^{\mathrm{aff}}(j_1;aq^{-j_2},n-j_2;q).
\]
Applying the backward shift relation \eqref{aff-backward} with
\[
i\mapsto i_1,\qquad
j\mapsto j_2-1,\qquad
a\mapsto aq^{-i_2-1},\qquad
n\mapsto n-i_2-1,
\]
we obtain
\[
\begin{aligned}
K_{i_1}^{\mathrm{aff}}(j_2;aq^{-i_2},n-i_2;q)
&=
q^{i_1}
K_{i_1}^{\mathrm{aff}}(j_2-1;aq^{-i_2-1},n-i_2-1;q)\\
&\quad
-
q^{i_1-1}
K_{i_1-1}^{\mathrm{aff}}(j_2-1;aq^{-i_2-1},n-i_2-1;q).
\end{aligned}
\]
Substituting this into the definition above and rewriting the two
resulting terms by Definition~\ref{def:biaff-qK}, with the parameters
$aq^{-1}$ and $n-1$, gives
\[
\begin{aligned}
K_{i_1,i_2}^{\mathrm{biaff}}(j_1,j_2;a,n;q)
&=
q^{i_1+2i_2}
K_{i_1,i_2}^{\mathrm{biaff}}(j_1,j_2-1;aq^{-1},n-1;q)\\
&\quad
-
q^{i_1+2i_2-1}
K_{i_1-1,i_2}^{\mathrm{biaff}}(j_1,j_2-1;aq^{-1},n-1;q).
\end{aligned}
\]
\end{proof}

\begin{lemma}
\label{lem:biaff_RR1}
Assume that $j_1\geq 1$. For
$(i_1,i_2),(j_1,j_2)\in I_n$, the following relation holds:
\[
\begin{aligned}
K_{i_1,i_2}^{\mathrm{biaff}}(j_1,j_2;a,n;q)
&=
q^{2i_1+3i_2}
K_{i_1,i_2}^{\mathrm{biaff}}(j_1-1,j_2;aq^{-1},n-1;q)\\
&\quad
+
q^{i_1+2i_2-1}
\bigl(q^n+a-q^{i_1+i_2}-q^{i_1+i_2-1}\bigr)
K_{i_1-1,i_2}^{\mathrm{biaff}}(j_1-1,j_2;aq^{-1},n-1;q)\\
&\quad
+
q^{i_2-1}
\bigl(q^n-q^{i_1+i_2-1}\bigr)
\bigl(a-q^{i_1+i_2-1}\bigr)
K_{i_1-2,i_2}^{\mathrm{biaff}}(j_1-1,j_2;aq^{-1},n-1;q)\\
&\quad
-
aq^{n+i_2-2}
K_{i_1,i_2-1}^{\mathrm{biaff}}(j_1-1,j_2;aq^{-1},n-1;q).
\end{aligned}
\]
\end{lemma}

\begin{proof}
If $i_2+j_2>n$, then the left-hand side is zero by the convention
following Definition~\ref{def:biaff-qK}. On the right-hand side, the
first three terms are zero since
\[
i_2+j_2>n-1,
\]
when they are regarded as bivariate affine $q$-Krawtchouk polynomials
with size $n-1$. The last term is also zero, since
\[
(i_2-1)+j_2>n-1
\]
is equivalent to $i_2+j_2>n$. Hence the relation is trivial in this
case. We may therefore assume that $i_2+j_2\le n$.

By Definition~\ref{def:biaff-qK},
\[
K_{i_1,i_2}^{\mathrm{biaff}}(j_1,j_2;a,n;q)
=
a^{i_2}q^{i_2(n-i_2)}
K_{i_1}^{\mathrm{aff}}(j_2;aq^{-i_2},n-i_2;q)
K_{i_2}^{\mathrm{aff}}(j_1;aq^{-j_2},n-j_2;q).
\]
Applying the backward shift relation \eqref{aff-backward} with
\[
i\mapsto i_2,\qquad
j\mapsto j_1-1,\qquad
a\mapsto aq^{-j_2-1},\qquad
n\mapsto n-j_2-1,
\]
we get
\[
\begin{aligned}
K_{i_2}^{\mathrm{aff}}(j_1;aq^{-j_2},n-j_2;q)
&=
q^{i_2}
K_{i_2}^{\mathrm{aff}}(j_1-1;aq^{-j_2-1},n-j_2-1;q)\\
&\quad
-
q^{i_2-1}
K_{i_2-1}^{\mathrm{aff}}(j_1-1;aq^{-j_2-1},n-j_2-1;q).
\end{aligned}
\]
Thus
\[
\begin{aligned}
K_{i_1,i_2}^{\mathrm{biaff}}(j_1,j_2;a,n;q)
&=
a^{i_2}q^{i_2(n-i_2+1)}
K_{i_1}^{\mathrm{aff}}(j_2;aq^{-i_2},n-i_2;q)\\
&\quad\times
K_{i_2}^{\mathrm{aff}}(j_1-1;aq^{-j_2-1},n-j_2-1;q)\\
&\quad
-
a^{i_2}q^{i_2(n-i_2+1)-1}
K_{i_1}^{\mathrm{aff}}(j_2;aq^{-i_2},n-i_2;q)\\
&\quad\times
K_{i_2-1}^{\mathrm{aff}}(j_1-1;aq^{-j_2-1},n-j_2-1;q).
\end{aligned}
\]

By Definition~\ref{def:biaff-qK}, with parameter $aq^{-1}$ and size $n-1$, the second term is
\[
-
aq^{n+i_2-2}
K_{i_1,i_2-1}^{\mathrm{biaff}}(j_1-1,j_2;aq^{-1},n-1;q),
\]
which is the last term in the desired recurrence. It remains to rewrite the first term.

If $i_2+j_2=n$, then $i_2>n-j_2-1$. Hence
\[
K_{i_2}^{\mathrm{aff}}(j_1-1;aq^{-j_2-1},n-j_2-1;q)
=0
\]
by the convention for the affine $q$-Krawtchouk polynomials. Thus the
first term vanishes.
On the other hand, in the desired recurrence, the first three terms on
the right-hand side are also zero. Indeed, these terms are
\[
\begin{gathered}
K_{i_1,i_2}^{\mathrm{biaff}}(j_1-1,j_2;aq^{-1},n-1;q),\quad
K_{i_1-1,i_2}^{\mathrm{biaff}}(j_1-1,j_2;aq^{-1},n-1;q),\\
K_{i_1-2,i_2}^{\mathrm{biaff}}(j_1-1,j_2;aq^{-1},n-1;q).
\end{gathered}
\]
Each has size $n-1$ with second degree index $i_2$ and second
argument $j_2$. Since
\[
i_2+j_2=n>n-1,
\]
all three vanish by the convention following
Definition~\ref{def:biaff-qK}.
Hence, when $i_2+j_2=n$, only the rewritten second term remains, and it gives the desired recurrence.

We may therefore assume that $i_2+j_2<n$.
We use Lemma~\ref{lem:aff-reduction} with
\[
i\mapsto i_1,\qquad
j\mapsto j_2, \qquad
a\mapsto aq^{-i_2-1},
\qquad
n\mapsto n-i_2-1.
\]
This gives
\[
\begin{aligned}
K_{i_1}^{\mathrm{aff}}(j_2;aq^{-i_2},n-i_2;q)
&=
q^{2i_1}
K_{i_1}^{\mathrm{aff}}(j_2;aq^{-i_2-1},n-i_2-1;q)\\
&\quad
+
q^{i_1}
\bigl(q^{n-i_2-1}+aq^{-i_2-1}-q^{i_1-1}-q^{i_1-2}\bigr)\\
&\quad\qquad\times
K_{i_1-1}^{\mathrm{aff}}(j_2;aq^{-i_2-1},n-i_2-1;q)\\
&\quad
+
q
\bigl(q^{n-i_2-1}-q^{i_1-2}\bigr)
\bigl(aq^{-i_2-1}-q^{i_1-2}\bigr)\\
&\quad\qquad\times
K_{i_1-2}^{\mathrm{aff}}(j_2;aq^{-i_2-1},n-i_2-1;q).
\end{aligned}
\]
Substituting this expression into the first term above and rewriting
the three resulting terms by Definition~\ref{def:biaff-qK}, with
parameters $aq^{-1}$ and $n-1$, gives
\[
\begin{aligned}
& q^{2i_1+3i_2}
K_{i_1,i_2}^{\mathrm{biaff}}(j_1-1,j_2;aq^{-1},n-1;q)\\
&\quad
+
q^{i_1+2i_2-1}
\bigl(q^n+a-q^{i_1+i_2}-q^{i_1+i_2-1}\bigr)
K_{i_1-1,i_2}^{\mathrm{biaff}}(j_1-1,j_2;aq^{-1},n-1;q)\\
&\quad
+
q^{i_2-1}
\bigl(q^n-q^{i_1+i_2-1}\bigr)
\bigl(a-q^{i_1+i_2-1}\bigr)
K_{i_1-2,i_2}^{\mathrm{biaff}}(j_1-1,j_2;aq^{-1},n-1;q).
\end{aligned}
\]
Combining this with the contribution of the second term computed above,
we obtain the right-hand side of the desired recurrence.
\end{proof}

We record the following boundary vanishing property at the specializations $a=q^d$.
It will be used to justify the boundary terms in the induction proof of Theorem \ref{main}.

\begin{lemma}
\label{lem:biaff-boundary-vanishing}
Let $0\le d\le n$, and let $(i_1,i_2)\in I_n$ and $(j_1,j_2)\in I_d$.
If $i_1+i_2>d$, then
\[
K^{\mathrm{biaff}}_{i_1,i_2}(j_1,j_2;q^d,n;q)=0 .
\]
\end{lemma}
\begin{proof}
We may assume that $i_2+j_2\le n$.
Indeed, if $i_2+j_2>n$, then the bivariate polynomial is zero by the convention following Definition~\ref{def:biaff-qK}, since the second affine factor has degree index $i_2$ larger than its size $n-j_2$.

By Definition~\ref{def:biaff-qK}, apart from the scalar prefactor, it is the product of the two one-variable affine $q$-Krawtchouk polynomials
\[
K^{\mathrm{aff}}_{i_1}(j_2;q^{d-i_2},n-i_2;q), \qquad 
K^{\mathrm{aff}}_{i_2}(j_1;q^{d-j_2},n-j_2;q). 
\]

Assume first that $i_2+j_2\le d$.
Then $j_2\le d-i_2$.
Since $i_1+i_2>d$ and $i_1+i_2 \le n$, we also have $d-i_2<i_1 \le n - i_2$.
Thus $0 \le j_2\le d-i_2<i_1\le n - i_2$.
By Lemma~\ref{lem:aff-boundary-vanishing}, applied with
\[
i \mapsto i_1, \qquad j \mapsto j_2, \qquad d \mapsto d-i_2, \qquad n \mapsto n - i_2,
\]
we obtain 
\[
K^{\mathrm{aff}}_{i_1}(j_2;q^{d-i_2},n-i_2;q)=0.
\]

It remains to consider the case $d < i_2+j_2 \le n$.
Then $d-j_2 < i_2 \le n - j_2$.
Since $j_1 + j_2 \le d$, we also have $j_1 \le d - j_2$.
Thus $0 \le j_1 \le d - j_2 < i_2 \le n - j_2$.
By Lemma~\ref{lem:aff-boundary-vanishing}, applied with
\[
i \mapsto i_2, \qquad j \mapsto j_1, \qquad d \mapsto d-j_2, \qquad n \mapsto n - j_2,
\]
we obtain
\[
K^{\mathrm{aff}}_{i_2}(j_1;q^{d-j_2},n-j_2;q)=0.
\]
Therefore, in all cases, one of the two factors vanishes.
Consequently, $K^{\mathrm{biaff}}_{i_1,i_2}(j_1,j_2;q^d,n;q)=0$.
\end{proof}

\section{The association scheme on matrix spaces}

Let $R=\operatorname{GR}(p^2,r)$ be the Galois ring of characteristic
$p^2$, and put $q=p^r$.
Fix $d,n \in \mathbb{Z}$ with $1 \leq d \leq n$ and let
\[
X=\operatorname{Mat}_{d\times n}(R)
\]
be the additive group of $d\times n$ matrices over $R$.

Every matrix $x\in X$ admits a Smith normal form over $R$; see, for example, \cite[Chapter~15]{B}.
We say
that $x$ has type $(i,j)$ if its Smith normal form has $i$
invariant factors equal to $p$, $j$ invariant factors equal to
$1$, and all remaining invariant factors equal to $0$. Let
\[
I_d
=
\{(i,j)\in\mathbb{Z}_{\geq0}^2\mid i+j\leq d\}.
\]
For each $(i,j)\in I_d$, define a relation $\mathcal{R}_{i,j}$ on
$X$ by
\[
(x,y)\in \mathcal{R}_{i,j}
\quad\Longleftrightarrow\quad
y-x \text{ has type }(i,j).
\]

\begin{proposition}
The pair
\[
\left(X,\{\mathcal{R}_{i,j}\}_{(i,j)\in I_d}\right)
\]
is a symmetric association scheme.
\end{proposition}

\begin{proof}
Let
\[
G=\operatorname{GL}_d(R)\times\operatorname{GL}_n(R)
\]
act on $X$ by
\[
(g_1,g_2)\cdot x=g_1 x g_2^{-1}.
\]
The $G$-orbits on $X$ are precisely the Smith type classes
\[
X_{i,j}=\{x\in X\mid x\text{ has type }(i,j)\}.
\]
Indeed, the Smith type is invariant under invertible row and column
operations, and two matrices over $R$ have the same Smith type if and
only if they are equivalent under such operations.

Now let $X\rtimes G$ act on $X$, where $X$ acts by translations
and $G$ acts as above. This action is transitive, and its orbitals are
\[
\{(x,y)\in X\times X\mid y-x\in X_{i,j}\},
\qquad (i,j)\in I_d.
\]
By the definition of $\mathcal{R}_{i,j}$, these orbitals are precisely
the relations $\mathcal{R}_{i,j}$. 
By the standard orbital
construction of association schemes from transitive group actions,
these relations form a Schurian association scheme; see, for example,
\cite[Example~2.3]{BBIT}.

It remains to check symmetry.
Since $-1$ is a unit in $R$, matrices $z$ and $-z$ have the same Smith type.
Therefore
\[
(x,y)\in \mathcal{R}_{i,j}
\quad\Longleftrightarrow\quad
y-x\in X_{i,j}
\quad\Longleftrightarrow\quad
x-y\in X_{i,j}
\quad\Longleftrightarrow\quad
(y,x)\in \mathcal{R}_{i,j}.
\]
Thus every relation is symmetric.
\end{proof}

Throughout the remainder of the paper, affine and bivariate affine
$q$-Krawtchouk polynomials are understood in the regularized sense
introduced in Sections~2 and~3. We now specialize the polynomial
parameter to
\[
a=q^d.
\]
By the regularization introduced in Section~2, this specialization is well defined.
When $d<n$, it is excluded from the generic hypergeometric definition.

\begin{theorem}
\label{main}
With the notation above, let $P$ be the first eigenmatrix of the
symmetric association scheme
\[
\mathfrak X=\left(X,\{\mathcal{R}_{i,j}\}_{(i,j)\in I_d}\right).
\]
We index the rows of $P$ by the character types $(s,t)\in I_d$ described in Section~6 and the columns by the relation types $(i,j)\in I_d$.
Then
\[
P_{(s,t),(i,j)} = K_{i,j}^{\mathrm{biaff}}(s,t;q^d,n;q), \qquad
(i,j),(s,t)\in I_d.
\]
Here $K^{\mathrm{biaff}}$ denotes the regularized bivariate affine
$q$-Krawtchouk polynomial defined in Definition~\ref{def:biaff-qK}.
\end{theorem}

\begin{example}
\label{ex:mat22-z4}
Let $R=\mathbb{Z}_4=\operatorname{GR}(2^2,1)$.
Then $p=2$, $q=2$, and we take $d=n=2$.
We use the convention that type $(i,j)$ means that
the Smith normal form has $i$ invariant factors equal to $p$,
$j$ invariant factors equal to $1$,
and all remaining invariant factors equal to $0$.

In this case
\[
I_2=\{(0,0),(1,0),(0,1),(2,0),(1,1),(0,2)\}.
\]
Theorem~\ref{main} gives
\[
P_{(s,t),(i,j)} = K^{\mathrm{biaff}}_{i,j}(s,t;4,2;2).
\]
We index the rows by the character types $(s,t)$ and the columns by the relation types $(i,j)$,
using the order
\[
(0,0),(1,0),(0,1),(2,0),(1,1),(0,2)
\]
for both.
Direct substitution into the formula for $K^{\mathrm{biaff}}$ gives
\[
P =
\begin{pmatrix}
1 & 9 & 72 & 6 & 72 & 96\\
1 & 9 & 8 & 6 & 8 & -32\\
1 & 1 & 8 & -2 & -8 & 0\\
1 & 9 & -24 & 6 & -24 & 32\\
1 & 1 & -8 & -2 & 8 & 0\\
1 & -3 & 0 & 2 & 0 & 0
\end{pmatrix}.
\]

With this ordering, the first row, corresponding to the trivial character type $(0,0)$, records the cardinalities of the type classes:
\[
1,\ 9,\ 72,\ 6,\ 72,\ 96.
\]
For example, the $((0,0),(0,1))$-entry $72$ means
that there are $72$ matrices of type $(0,1)$.
These cardinalities sum to
\[
1+9+72+6+72+96=256=
\left| \operatorname{Mat}_{2\times 2}(\mathbb{Z}_4)\right|.
\]

Although the general expression for the eigenvalues as character sums
will be discussed in Section~6, we verify one nontrivial entry directly in this example.
Take
\[
z=
\begin{pmatrix}
2&0\\
0&0
\end{pmatrix},
\]
of type $(1,0)$, and consider the corresponding additive character
\[ 
x\longmapsto\psi(\operatorname{tr}(zx^t)),
\]
where $\psi$ is the standard additive character
\[
\psi:\mathbb{Z}_4\longrightarrow \mathbb C^\times,
\qquad
\psi(r)=\exp(2\pi \sqrt{-1}\,r/4),
\]
and $\operatorname{tr}$ denotes the usual matrix trace on $\operatorname{Mat}_{2\times2}(\mathbb{Z}_4)$.
For
$x=(x_{ab})\in X$, we have
\[
\psi(\operatorname{tr}(zx^t))
=
\psi(2x_{11})
=
(-1)^{x_{11}}.
\]
Among the $72$ matrices of type $(0,1)$, there are $40$ with
$x_{11}$ even and $32$ with $x_{11}$ odd. Hence
\[
\sum_{x\in X_{0,1}}\psi(\operatorname{tr}(zx^t))
=
40-32
=
8.
\]
This agrees with the $((1,0),(0,1))$-entry of $P$.
\end{example}

\section{Counting matrices}

Throughout this section, we keep the notation of Section~4. Thus
\[
R=\operatorname{GR}(p^2,r),\qquad q=p^r,\qquad
I_d=\{(i,j)\in\mathbb Z_{\geq0}^2\mid i+j\leq d\}, \qquad
X=\operatorname{Mat}_{d\times n}(R),
\]
with $d\leq n$. We also write
\[
\mathbb{F}_q=R/pR
\]
for the residue field.
For $(i,j)\in I_d$, set
\[
X_{i,j}
=
\{x\in X\mid x \text{ has type }(i,j)\}.
\]
Here, as in Section~4, type $(i,j)$ means that the Smith normal form
has $i$ invariant factors equal to $p$, $j$ invariant factors
equal to $1$, and all remaining invariant factors equal to $0$.

Fix a set of representatives in $R$ for the residue classes modulo $pR$,
chosen so that the representatives of $0$ and $1$ are $0$ and $1$, respectively.
We use this choice throughout the section.
For $a\in\mathbb F_q$, denote its representative by $\widetilde a$.
Thus tildes will indicate the fixed representatives in $R$, while we continue to write $0$ and $1$ for their representatives.
Using these representatives, every element $\rho \in R$ can be written uniquely
in the form
\[
\rho =\widetilde{\alpha}+p\widetilde{\beta},
\qquad \alpha,\beta\in\mathbb{F}_q.
\]
Indeed, $\alpha$ is the image of $\rho$ in $R/pR$, so that $\rho-\widetilde{\alpha}\in pR$.
Every element of $pR$ can be written uniquely as $p\widetilde{\beta}$ with $\beta\in\mathbb F_q$.

\subsection{Type classes}

\begin{proposition}
\label{prop:type-class-cardinality}
For $(i,j)\in I_d$, one has
\[
|X_{i,j}|
=
\begin{bmatrix} n\\ j \end{bmatrix}_q
\begin{bmatrix} n-j\\ i \end{bmatrix}_q
q^{j(n+d-i-j)}
\prod_{\nu=0}^{i+j-1}(q^d-q^\nu).
\]
\end{proposition}

\begin{proof}
For $\alpha,\beta\in \operatorname{Mat}_{d\times n}(\mathbb F_q)$,
let $\widetilde\alpha,\widetilde\beta$ denote their entrywise lifts to $R$ with respect to the fixed representatives.
Then every $x\in X$ can be written uniquely in the form
\[
x=\widetilde{\alpha}+p\widetilde{\beta},
\qquad
\alpha,\beta\in \operatorname{Mat}_{d\times n}(\mathbb{F}_q).
\]
Here $\alpha$ is the reduction of $x$ modulo $p$.

We first determine the type of $x$. Suppose that $\operatorname{rank}_{\mathbb{F}_q}(\alpha)=j$.
Choose
$g_1\in \operatorname{GL}_d(\mathbb{F}_q)$ and $g_2\in \operatorname{GL}_n(\mathbb{F}_q)$
such that
\[
g_1\alpha g_2
=
\alpha_0
:=
\begin{pmatrix}
I_j&0\\
0&0
\end{pmatrix}.
\]
Equivalently, this amounts to choosing a complement
$U$ of $\ker(\alpha)$ in $\mathbb F_q^n$ and a complement
$C$ of $\operatorname{im}(\alpha)$ in $\mathbb F_q^d$, together with bases
adapted to these decompositions such that
$\alpha|_U:U\to \operatorname{im}(\alpha)$ is represented by the identity matrix.

Lift $g_1$ and $g_2$ to invertible matrices
$\widetilde g_1\in \operatorname{GL}_d(R)$ and $\widetilde g_2\in \operatorname{GL}_n(R)$.
Since
\[
\widetilde g_1\widetilde{\alpha}\widetilde g_2
\equiv
g_1\alpha g_2
=
\alpha_0
\pmod p,
\]
there exists a matrix
$\beta'\in \operatorname{Mat}_{d\times n}(\mathbb{F}_q)$
such that
\[
\widetilde g_1\widetilde{\alpha}\widetilde g_2
=
\alpha_0+p\widetilde{\beta'}.
\]
Moreover, since $p^2=0$, we have
\[
p\widetilde g_1\widetilde{\beta}\widetilde g_2
=
p\,\widetilde{g_1\beta g_2}.
\]
Also, the matrix
\[
\widetilde{\beta'} +\widetilde{g_1\beta g_2} -\widetilde{\beta'+g_1\beta g_2}
\]
has entries in $pR$, since its reduction modulo $p$ is zero.
Hence
\[
p\widetilde{\beta'}+ p\widetilde{g_1\beta g_2}=p\widetilde{(\beta'+g_1\beta g_2)}.
\]
Therefore
\[
\widetilde g_1x\widetilde g_2
=
\widetilde g_1(\widetilde{\alpha}+p\widetilde{\beta})\widetilde g_2
=
\alpha_0
+
p\widetilde{\gamma},
\]
where
$\gamma
=
\beta'+g_1\beta g_2
\in \operatorname{Mat}_{d\times n}(\mathbb{F}_q)$.

Write $\gamma$ in block form compatible with
$\alpha_0=\operatorname{diag}(I_j,0)$,
namely
\[
\gamma=
\begin{pmatrix}
\gamma_{11}&\gamma_{12}\\
\gamma_{21}&\gamma_{22}
\end{pmatrix},
\]
where
$\gamma_{22}\in
\operatorname{Mat}_{(d-j)\times(n-j)}(\mathbb{F}_q)$.
Equivalently, the block $\gamma_{22}$ represents the induced map
\[
\ker(\alpha)\longrightarrow
\mathbb F_q^d/\operatorname{im}(\alpha)
\]
determined by the $p$-part after reducing modulo $p$, with respect
to the chosen complements. 
Then
\[
\widetilde g_1x\widetilde g_2
=
\begin{pmatrix}
I_j+p\widetilde{\gamma}_{11}
&
p\widetilde{\gamma}_{12}\\
p\widetilde{\gamma}_{21}
&
p\widetilde{\gamma}_{22}
\end{pmatrix}.
\]
Since $p^2=0$, the block
$I_j+p\widetilde{\gamma}_{11}$ is invertible over $R$, with inverse
$I_j-p\widetilde{\gamma}_{11}$. Hence elementary row and column
operations over $R$ reduce the above matrix to
\[
\begin{pmatrix}
I_j&0\\
0&p\widetilde{\gamma}_{22}
\end{pmatrix}.
\]
Indeed, since the upper-left block is invertible, standard block Gaussian
elimination over $R$ eliminates the upper-right and lower-left blocks.
The resulting lower-right block differs from
$p\widetilde{\gamma}_{22}$ by terms divisible by $p^2$, hence is equal to
$p\widetilde{\gamma}_{22}$.

Now let
$h=\operatorname{rank}_{\mathbb{F}_q}(\gamma_{22})$.
Then, over $\mathbb{F}_q$, the matrix $\gamma_{22}$ can be reduced by
invertible row and column operations to
\[
\begin{pmatrix}
I_h&0\\
0&0
\end{pmatrix}.
\]
Lifting these operations to $R$, the block $p\widetilde{\gamma}_{22}$
is equivalent over $R$ to
\[
\begin{pmatrix}
pI_h&0\\
0&0
\end{pmatrix}.
\]
Thus the Smith normal form of $x$ has $j$ invariant factors equal to
$1$, $h$ invariant factors equal to $p$, and all remaining
invariant factors equal to $0$. Therefore, for fixed $\alpha$ of
rank $j$, the matrix $x$ has type $(i,j)$ if and only if
$\operatorname{rank}_{\mathbb{F}_q}(\gamma_{22})=i$.

We now count such matrices. The number of choices of
$\alpha\in \operatorname{Mat}_{d\times n}(\mathbb{F}_q)$ with rank
$j$ is
\[
\begin{bmatrix} n\\ j \end{bmatrix}_q
\prod_{\nu=0}^{j-1}(q^d-q^\nu).
\]
Fix such an $\alpha$, and also fix $g_1,g_2$ as above. Then
$\beta'$ is fixed. As $\beta$ varies over
$\operatorname{Mat}_{d\times n}(\mathbb{F}_q)$, the matrix
\[
\gamma=\beta'+g_1\beta g_2
\]
also varies over all of
$\operatorname{Mat}_{d\times n}(\mathbb{F}_q)$. In particular,
$\gamma_{22}$ varies over all matrices in
$\operatorname{Mat}_{(d-j)\times(n-j)}(\mathbb{F}_q)$.
Moreover, for each fixed value of $\gamma_{22}$, the remaining blocks
$\gamma_{11},\gamma_{12},\gamma_{21}$
are arbitrary. Hence each value of $\gamma_{22}$ occurs
\[
q^{j^2+j(n-j)+(d-j)j}
=
q^{j(n+d-j)}
\]
times.
Therefore, for fixed $\alpha$, the number of $\beta$ such that
$\operatorname{rank}_{\mathbb{F}_q}(\gamma_{22})=i$ is
\[
q^{j(n+d-j)}
\begin{bmatrix} n-j\\ i \end{bmatrix}_q
\prod_{\nu=0}^{i-1}(q^{d-j}-q^\nu).
\]
Combining this with the number of possible $\alpha$, we get
\[
|X_{i,j}|
=
\begin{bmatrix} n\\ j \end{bmatrix}_q
\prod_{\nu=0}^{j-1}(q^d-q^\nu)
\cdot
q^{j(n+d-j)}
\cdot
\begin{bmatrix} n-j\\ i \end{bmatrix}_q
\prod_{\nu=0}^{i-1}(q^{d-j}-q^\nu).
\]
Using
\[
\prod_{\nu=0}^{i-1}(q^{d-j}-q^\nu)
=
q^{-ij}\prod_{\nu=j}^{i+j-1}(q^d-q^\nu),
\]
we obtain
\[
|X_{i,j}|
=
\begin{bmatrix} n\\ j \end{bmatrix}_q
\begin{bmatrix} n-j\\ i \end{bmatrix}_q
q^{j(n+d-i-j)}
\prod_{\nu=0}^{i+j-1}(q^d-q^\nu).
\]
\end{proof}

\subsection{Definition and reduction of $N_{i,j}(\rho;k,l)$}

For the rest of this section, we assume $d\geq 2$. We study how the type changes when one
row and one column are added.

\begin{lemma}
\label{lem:N-well-defined}
Fix $(i,j)\in I_d$, $\rho\in R$, and
$y\in\operatorname{Mat}_{(d-1)\times(n-1)}(R)$. Then the number of
pairs $(u,v)\in R^{n-1}\times R^{d-1}$ such that
\[
\begin{pmatrix}
\rho & u^t\\
v & y
\end{pmatrix}
\in X_{i,j}
\]
depends only on the type of $\rho$ and the type of $y$, and not on
the particular choices $\rho$ and $y$.
\end{lemma}

\begin{proof}
Let $y$ and $y'$ be two matrices of the same type. Then there exist
$g_1 \in\operatorname{GL}_{d-1}(R)$ and $g_2 \in\operatorname{GL}_{n-1}(R)$
such that
$y'=g_1 y g_2$.
The change of variables
$u\mapsto g_2^t u$, $v\mapsto g_1v$
is a bijection on $R^{n-1}\times R^{d-1}$.
Moreover,
\[
\begin{pmatrix}
1 & 0\\
0 & g_1
\end{pmatrix}
\begin{pmatrix}
\rho & u^t\\
v & y
\end{pmatrix}
\begin{pmatrix}
1 & 0\\
0 & g_2
\end{pmatrix}
=
\begin{pmatrix}
\rho & (g_2^tu)^t\\
g_1v & g_1yg_2
\end{pmatrix}.
\]
Hence the Smith type is preserved.

If $\sigma=a\rho b$ with $a,b\in R^\times$, then multiplication on the
left and right by the block diagonal matrices gives
\[
\begin{pmatrix}
a&0\\
0&I_{d-1}
\end{pmatrix}
\begin{pmatrix}
\rho&u^t\\
v&y
\end{pmatrix}
\begin{pmatrix}
b&0\\
0&I_{n-1}
\end{pmatrix}
=
\begin{pmatrix}
\sigma&a u^t\\
v b&y
\end{pmatrix}.
\]
This gives a bijection on the choices of $(u,v)$ and preserves the
Smith type.

Therefore the number of such pairs $(u,v)$ depends only on the type of
$\rho$ and the type of $y$.
\end{proof}

By Lemma~\ref{lem:N-well-defined}, for $(k,l)\in I_{d-1}$, we may define
\[
N_{i,j}(\rho;k,l)
=
\#\left\{
(u,v)\in R^{n-1}\times R^{d-1}
\;\middle|\;
\begin{pmatrix}
\rho & u^t\\
v & y
\end{pmatrix}
\in X_{i,j}
\right\},
\]
where $y$ is any fixed matrix of type $(k,l)$.
Thus, in what follows, we regard $\rho$ as taking values in
$\{0,p,1\}$, representing the zero, nonzero nonunit, and unit cases,
respectively.
By the same lemma,
we may take $y$ to be the standard representative
\[
y=
\operatorname{diag}(
\underbrace{p,\ldots,p}_{k},
\underbrace{1,\ldots,1}_{l},
0,\ldots,0),
\]
where the remaining diagonal entries are zero and the matrix has size $(d-1)\times(n-1)$.

With respect to this block decomposition, write
\[
u=(u_1,u_2,u_3),
\qquad
v=(v_1,v_2,v_3),
\]
where
\[
u_1,v_1\in R^k,\qquad
u_2,v_2\in R^l, \qquad
u_3\in R^{n-k-l-1},
\qquad
v_3\in R^{d-k-l-1}.
\]
Here and below, if one of the dimensions $k$, $l$, $n-k-l-1$, or $d-k-l-1$ is zero,
the corresponding module is understood to be the zero module;
in particular, any case requiring a nonzero vector in that module is void.
Then we write uniquely
\[
u_1=\widetilde{\alpha}_1+p\widetilde{\beta}_1,\qquad
v_1=\widetilde{\xi}_1+p\widetilde{\eta}_1, \qquad
u_2=\widetilde{\alpha}_2+p\widetilde{\beta}_2,\qquad
v_2=\widetilde{\xi}_2+p\widetilde{\eta}_2,
\]
where
\[
\alpha_1,\beta_1,\xi_1,\eta_1\in\mathbb{F}_q^k,
\qquad
\alpha_2,\beta_2,\xi_2,\eta_2\in\mathbb{F}_q^l.
\]
The components $p\widetilde{\beta}_1$ and $p\widetilde{\eta}_1$ can be eliminated using the $pI_k$ block,
whereas $\widetilde{\alpha}_1$ and $\widetilde{\xi}_1$ remain.
Similarly, the entire $u_2$ and $v_2$ components can be eliminated using the invertible
block $I_l$. Carrying out these elementary row and column operations
over $R$, the block matrix is reduced to
\[
M =
\begin{pmatrix}
\rho' & \widetilde{\alpha}_1^t & 0 & u_3^t\\
\widetilde{\xi}_1 & pI_k & 0 & 0\\
0 & 0 & I_l & 0\\
v_3 & 0 & 0 & 0
\end{pmatrix},
\]
where
\[
\rho'
=
\rho
-
(
\widetilde{\beta}_1^t\widetilde{\xi}_1
+
\widetilde{\alpha}_1^t\widetilde{\eta}_1
+
\widetilde{\alpha}_2^t\widetilde{\xi}_2
)
-
p( 
\widetilde{\beta}_1^t\widetilde{\eta}_1
+
\widetilde{\beta}_2^t\widetilde{\xi}_2
+
\widetilde{\alpha}_2^t\widetilde{\eta}_2
).
\]

The invertible block $I_l$ can be separated and contributes $l$ invariant factors equal to $1$.
After separating this block, the Smith type of the remaining matrix is determined by $\rho'$, $\alpha_1$, $\xi_1$, $u_3$, and $v_3$, together with the block $pI_k$.
The next subsection classifies the possible Smith types of the reduced matrix in terms of these data.
The classification in the next subsection depends only on the reduced data,
including $\rho'$, and not directly on the original value of $\rho$.
The dependence on $\rho$ will enter later through the distribution of $\rho'$.

\subsection{Type criteria}
\label{subsec:type-criteria}

For the reduced matrix above, put
\[
N=n-k-l-1,\qquad D=d-k-l-1.
\]
Thus $u_3\in R^N$ and $v_3\in R^D$.
For each of $u_3$ and $v_3$, we distinguish three cases:
\[
u_3=0,\qquad u_3\in(pR)^N\setminus\{0\},\qquad u_3\notin(pR)^N,
\]
and similarly
\[
v_3=0,\qquad v_3\in(pR)^D\setminus\{0\},\qquad v_3\notin(pR)^D.
\]
Since transposition preserves Smith type, we use the transposed cases when appropriate.
Throughout this subsection, the conditions on $\alpha_1$ and $\xi_1$,
such as $\alpha_1=0$, $\xi_1\neq0$, and
$\alpha_1^t\xi_1=0$, are understood over $\mathbb F_q$.
When these vectors occur as entries of matrices over $R$, we use their
fixed lifts and write them with tildes.

\begin{lemma}
\label{lem:type-case-u3-v3-zero}
Assume that $u_3=0$ and $v_3=0$.
Then the reduced block matrix $M$ has the following types.
\begin{enumerate}
\item If $\alpha_1=0$ and $\xi_1=0$,
then
\[
\operatorname{type}(M)=
\begin{cases}
(k,l), & \rho'=0,\\
(k+1,l), & \rho'\in pR\setminus\{0\},\\
(k,l+1), & \rho'\notin pR.
\end{cases}
\]

\item If exactly one of $\alpha_1$ and $\xi_1$ is nonzero,
then
\[
\operatorname{type}(M)=
\begin{cases}
(k-1,l+1), & \rho'\in pR,\\
(k,l+1), & \rho'\notin pR.
\end{cases}
\]

\item If $\alpha_1\neq 0$ and $\xi_1\neq 0$,
then
\[
\operatorname{type}(M)=
\begin{cases}
(k-2,l+2), & \alpha_1^t\xi_1=0,\\
(k-1,l+2), & \alpha_1^t\xi_1\neq0.
\end{cases}
\]
\end{enumerate}
\end{lemma}
\begin{proof}
Under the assumptions $u_3=0$ and $v_3=0$,
the last block row and the last block column of $M$ are zero.
In addition, the identity block $I_l$ contributes $l$ invariant factors equal to $1$.
Thus, it remains to determine the Smith type of
\[
M_0 =
\begin{pmatrix}
\rho' & \widetilde{\alpha}_1^t\\
\widetilde{\xi}_1 & pI_k
\end{pmatrix}.
\]

First suppose that $\alpha_1=0$ and $\xi_1=0$.
Then $M_0=\rho'\oplus pI_k$.
The block $pI_k$ contributes $k$ invariant factors equal to $p$.
The scalar block $\rho'$ contributes nothing if $\rho'=0$, one
invariant factor equal to $p$ if $\rho'\in pR\setminus\{0\}$, and
one invariant factor equal to $1$ if $\rho'\notin pR$. Hence $M$
has type
\[
\begin{cases}
(k,l), & \rho'=0,\\
(k+1,l), & \rho'\in pR\setminus\{0\},\\
(k,l+1), & \rho'\notin pR.
\end{cases}
\]

Next suppose that exactly one of $\alpha_1$ and $\xi_1$ is nonzero.
By transposition, it suffices to consider the case $\alpha_1=0$ and $\xi_1\neq 0$.
Using a unit entry of $\widetilde{\xi}_1$ as a pivot, we split off one invariant
factor equal to $1$. 
If $\rho'\in pR$, then $\rho'$ is eliminated by the
$\widetilde{\xi}_1$-pivot using a coefficient in $pR$, so no new nonzero entry
is created, since $p^2=0$. Thus one $p$-direction from $pI_k$ is
absorbed. Hence the type is $(k-1,l+1)$.
If $\rho'\notin pR$, then $\rho'$ itself is a unit pivot;
after splitting it off, the block $pI_k$ remains equivalent to
$pI_k$. 
Hence the type is $(k,l+1)$.
The case $\alpha_1\neq0,\xi_1=0$ follows by transposition.

It remains to consider the case $\alpha_1\neq0$ and $\xi_1\neq0$.
Using elementary row and column operations within the $pI_k$-block,
we may transform $M_0$ into the form
\[
M_0 \sim
\begin{pmatrix}
\rho' & 1 & 0\\
\widetilde{c} & p & 0\\
\widetilde{\eta}  & 0 & pI_{k-1}
\end{pmatrix},
\qquad
c=\alpha_1^t\xi_1\in\mathbb F_q,
\]
where $\eta\in\mathbb{F}_q^{k-1}$.
Here, after choosing coordinates so that $\alpha_1=(1,0,\ldots,0)$,
the vector $\xi_1$ has coordinates $(c,\eta)^t$.
Splitting off the unit pivot $1$, we obtain the remaining block
\[
M_1 =
\begin{pmatrix}
\widetilde{c}-p\rho' & 0\\
\widetilde{\eta}     & pI_{k-1}
\end{pmatrix}.
\]
If $c\neq0$, then $\widetilde c-p\rho'$ is a unit, so $M_1$
contributes one unit factor and $k-1$ $p$-factors. Hence the type is
$(k-1,l+2)$.
If $c=0$, then $\eta\neq0$ because $\xi_1\neq0$.
A unit entry of $\widetilde\eta$ gives another unit pivot.
Since $\widetilde c-p\rho'=-p\rho'\in pR$,
this entry can be eliminated using the unit pivot without creating further nonzero entries, because $p^2=0$.
The remaining $p$-part is $pI_{k-2}$, so $M$ has type $(k-2,l+2)$.
\end{proof}

\begin{lemma}
\label{lem:type-case-u-p-nonzero-v-zero}
Assume that $u_3\in (pR)^N\setminus\{0\}$ and $v_3=0$.
Then the reduced block matrix $M$ has the following types.

\begin{enumerate}
\item If $\alpha_1=0$ and $\xi_1=0$, then
\[
\operatorname{type}(M)=
\begin{cases}
(k+1,l), & \rho'\in pR,\\
(k,l+1), & \rho'\notin pR.
\end{cases}
\]

\item If $\alpha_1=0$ and $\xi_1\neq0$, then
\[
\operatorname{type}(M)=(k,l+1).
\]

\item If $\alpha_1\neq0$ and $\xi_1=0$, then
\[
\operatorname{type}(M)=
\begin{cases}
(k-1,l+1), & \rho'\in pR,\\
(k,l+1), & \rho'\notin pR.
\end{cases}
\]

\item If $\alpha_1\neq0$ and $\xi_1\neq0$, then
\[
\operatorname{type}(M)=
\begin{cases}
(k-2,l+2), & \alpha_1^t\xi_1=0,\\
(k-1,l+2), & \alpha_1^t\xi_1\neq0.
\end{cases}
\]
\end{enumerate}
\end{lemma}

\begin{proof}
Since $v_3=0$, the last block row is zero and may be omitted.
The block $I_l$ contributes $l$ invariant factors equal to $1$.
Thus it remains to determine the Smith type of
\[
M_0 =
\begin{pmatrix}
\rho' & \widetilde{\alpha}_1^t & u_3^t\\
\widetilde{\xi}_1 & pI_k & 0
\end{pmatrix}.
\]
We distinguish the four cases according to whether $\alpha_1$ and
$\xi_1$ are zero.

First suppose that $\alpha_1=0$ and $\xi_1=0$.
If $\rho'\notin pR$, then $\rho'$ is a unit. Using $\rho'$ as a
pivot, we split off one invariant factor equal to $1$. Since
$u_3\in(pR)^N$, the entries of $u_3$ are eliminated by this unit
pivot, and the block $pI_k$ remains equivalent to $pI_k$. Hence the
type is $(k,l+1)$.

If $\rho'\in pR$, then all entries in $M_0$ are contained in
$pR$. Write
\[
\rho'=p\widetilde a,\qquad
u_3=p\widetilde w,
\]
with $a\in\mathbb F_q$ and $w\in\mathbb F_q^N\setminus\{0\}$.
Since every entry of $M_0$ lies in $pR$, take its coefficient of $p$ entrywise.
The resulting matrix over $\mathbb F_q$ contains the block
\[
\begin{pmatrix}
a & 0 & w^t\\
0 & I_k & 0
\end{pmatrix},
\]
which has rank $k+1$, since $w\neq0$. Thus the original
matrix contributes $k+1$ invariant factors equal to $p$, and the
type is $(k+1,l)$.

Next suppose that $\alpha_1=0$ and $\xi_1\neq0$.
A unit entry of $\widetilde{\xi}_1$ can be used as a pivot, and hence one
invariant factor equal to $1$ is split off. This pivot absorbs one
$p$-direction from the block $pI_k$. On the other hand, since
$u_3\in(pR)^N\setminus\{0\}$,
the remaining first row contains a nonzero $p$-entry, which
contributes one new invariant factor equal to $p$. Thus the loss of
one $p$-factor from $pI_k$ is compensated by the new $p$-factor
coming from $u_3$. Therefore the type is $(k,l+1)$.

Finally suppose that
$\alpha_1\neq0$.
A unit entry of $\widetilde{\alpha}_1$ can be used as a pivot. Since
$u_3\in(pR)^N$, the coefficients used to eliminate the entries of
$u_3$ lie in $pR$. The pivot column has a $p$-entry in the
$pI_k$-block below the first row, so the entries created below the
$u_3$-block are multiples of $p^2$, hence vanish. Thus $u_3$ is
eliminated without creating new entries.

Therefore the remaining computation is the same as in the case
$u_3=0$, $v_3=0$, and $\alpha_1\neq0$ from
Lemma~\ref{lem:type-case-u3-v3-zero}. If $\xi_1=0$, this gives
\[
\operatorname{type}(M)=
\begin{cases}
(k-1,l+1), & \rho'\in pR,\\
(k,l+1), & \rho'\notin pR.
\end{cases}
\]
If $\xi_1\neq0$, it gives
\[
\operatorname{type}(M)=
\begin{cases}
(k-2,l+2), & \alpha_1^t\xi_1=0,\\
(k-1,l+2), & \alpha_1^t\xi_1\neq0.
\end{cases}
\]
\end{proof}

\begin{lemma}
\label{lem:type-case-u-zero-v-p-nonzero}
Assume that $u_3=0$ and $v_3\in(pR)^D\setminus\{0\}$.
Then the reduced block matrix $M$ has the following types.

\begin{enumerate}
\item If $\alpha_1=0$ and $\xi_1=0$,
then
\[
\operatorname{type}(M)=
\begin{cases}
(k+1,l), & \rho'\in pR,\\
(k,l+1), & \rho'\notin pR.
\end{cases}
\]

\item If $\alpha_1\neq0$ and $\xi_1=0$,
then
\[
\operatorname{type}(M)=(k,l+1).
\]

\item If $\alpha_1=0$ and $\xi_1\neq0$,
then
\[
\operatorname{type}(M)=
\begin{cases}
(k-1,l+1), & \rho'\in pR,\\
(k,l+1), & \rho'\notin pR.
\end{cases}
\]

\item If $\alpha_1\neq0$ and $\xi_1\neq0$,
then
\[
\operatorname{type}(M)=
\begin{cases}
(k-2,l+2), & \alpha_1^t\xi_1=0,\\
(k-1,l+2), & \alpha_1^t\xi_1\neq0.
\end{cases}
\]
\end{enumerate}
\end{lemma}

\begin{proof}
This follows from Lemma~\ref{lem:type-case-u-p-nonzero-v-zero} by transposition.
Indeed, transposition preserves Smith type and interchanges
\[
u_3\leftrightarrow v_3,\qquad
\alpha_1\leftrightarrow \xi_1,\qquad
N\leftrightarrow D.
\]
Thus the assumptions
\[
u_3\in(pR)^N\setminus\{0\},\qquad v_3=0
\]
in Lemma~\ref{lem:type-case-u-p-nonzero-v-zero} correspond exactly to
\[
u_3=0,\qquad v_3\in(pR)^D\setminus\{0\}
\]
here. Applying Lemma~\ref{lem:type-case-u-p-nonzero-v-zero} to $M^t$
gives the stated alternatives.
\end{proof}

\begin{lemma}
\label{lem:type-case-u-p-nonzero-v-p-nonzero}
Assume that
$u_3\in (pR)^N\setminus\{0\}$ and $v_3\in (pR)^D\setminus\{0\}$.
Then the reduced block matrix $M$ has the following types.

\begin{enumerate}
\item If $\alpha_1=0$ and $\xi_1=0$, then
\[
\operatorname{type}(M)=
\begin{cases}
(k+2,l), & \rho'\in pR,\\
(k,l+1), & \rho'\notin pR.
\end{cases}
\]

\item If exactly one of $\alpha_1$ and $\xi_1$ is nonzero, then
\[
\operatorname{type}(M)=(k,l+1).
\]

\item If $\alpha_1\neq0$ and $\xi_1\neq0$, then
\[
\operatorname{type}(M)=
\begin{cases}
(k-2,l+2), & \alpha_1^t\xi_1=0,\\
(k-1,l+2), & \alpha_1^t\xi_1\neq0.
\end{cases}
\]
\end{enumerate}
\end{lemma}

\begin{proof}
We distinguish the cases according to whether $\alpha_1$ and
$\xi_1$ are zero.

First suppose that
$\alpha_1=0$ and $\xi_1=0$.
Then the reduced matrix has the form
\[
M=
\begin{pmatrix}
\rho' & 0 & 0 & u_3^t\\
0 & pI_k & 0 & 0\\
0 & 0 & I_l & 0\\
v_3 & 0 & 0 & 0
\end{pmatrix}.
\]
If $\rho'\notin pR$, then $\rho'$ is a unit. Using $\rho'$ as a
pivot, we split off one invariant factor equal to $1$. Since
$u_3\in(pR)^N$ and $v_3\in(pR)^D$, the entries of $u_3$ and
$v_3$ can be eliminated using this unit pivot, and the block $pI_k$
remains equivalent to $pI_k$. Hence the type is $(k,l+1)$.

Now suppose that $\rho'\in pR$. 
Since $u_3\in(pR)^N$ and $v_3\in(pR)^D$, after separating the block $I_l$, every entry of the remaining matrix lies in $pR$.
Write
\[
\rho'=p\widetilde a,\qquad
u_3=p\widetilde w,\qquad
v_3=p\widetilde z,
\]
where
\[
a \in\mathbb{F}_q,\qquad
w\in\mathbb{F}_q^N\setminus\{0\},\qquad
z\in\mathbb{F}_q^D\setminus\{0\}.
\]
Taking the coefficients of $p$ entrywise,
the number of invariant factors equal to $p$ contributed by this part is the rank over $\mathbb F_q$ of
\[
\begin{pmatrix}
a & 0 & w^t\\
0 & I_k & 0\\
z & 0 & 0
\end{pmatrix}.
\]
The block $I_k$ contributes rank $k$. Since both
$w$ and $z$ are nonzero, the first row and
the first column contribute two additional independent directions.
Hence this matrix has rank $k+2$ over $\mathbb{F}_q$. Therefore the
type is $(k+2,l)$.

Next suppose that at least one of $\alpha_1$ and $\xi_1$ is
nonzero. It is enough to consider the case $\xi_1\neq0$, since the
case $\alpha_1\neq0$ and $\xi_1=0$ follows by transposition.
Since $v_3\in(pR)^D$, a unit pivot from $\widetilde{\xi}_1$ eliminates $v_3$
without changing the zero blocks; any entries created are multiples of
$p^2$, and hence vanish. Thus the remaining computation is the same
as in Lemma~\ref{lem:type-case-u-p-nonzero-v-zero} with $v_3=0$ and
$\xi_1\neq0$.

If $\alpha_1=0$, this gives type
\[
(k,l+1).
\]
If $\alpha_1\neq0$, this gives
\[
\operatorname{type}(M)=
\begin{cases}
(k-2,l+2), & \alpha_1^t\xi_1=0,\\
(k-1,l+2), & \alpha_1^t\xi_1\neq0.
\end{cases}
\]
The case $\alpha_1\neq0$ and $\xi_1=0$ follows by transposition.
\end{proof}

\begin{lemma}
\label{lem:type-case-u-nonp-v-zero}
Assume that $u_3\notin (pR)^N$ and $v_3=0$.
Then the reduced block matrix $M$ has the following types:
\[
\operatorname{type}(M)=
\begin{cases}
(k-1,l+2), & \text{if } \xi_1\neq 0,\\
(k,l+1), & \text{if } \xi_1=0.
\end{cases}
\]
\end{lemma}

\begin{proof}
Since $u_3\notin(pR)^N$, some entry of $u_3$ is a unit. By column
operations within the last $N$ columns, we may assume that
\[
u_3^t=(1,0,\ldots,0).
\]
The corresponding pivot column has zero entries outside the first row.
Hence, using this pivot column, we can eliminate all other entries in
the first row by column operations. In particular, $\rho'$ and
$\widetilde{\alpha}_1^t$ are eliminated. Thus one invariant factor equal to $1$
is split off.

Since $v_3=0$, after this unit factor has been split off, the only
possible remaining contribution from the first column is $\xi_1$.
The blocks $pI_k$ and $I_l$ are otherwise unchanged at this stage.

If $\xi_1\neq0$, then a unit entry of $\widetilde{\xi}_1$ can be used as a
pivot. This splits off another invariant factor equal to $1$. This
pivot uses one of the rows belonging to the $pI_k$-block, and hence
one invariant factor equal to $p$ from $pI_k$ is absorbed. Therefore the
remaining $p$-part is equivalent to $pI_{k-1}$, and the resulting
type is
\[
(k-1,l+2).
\]

If $\xi_1=0$, then no further invariant factor is created from the
first column. The block $pI_k$ remains unchanged, and the only new
unit factor is the one coming from the unit entry of $u_3$. Hence the
resulting type is
\[
(k,l+1).
\]
\end{proof}

\begin{lemma}
\label{lem:type-case-u-zero-v-nonp}
Assume that $u_3=0$ and $v_3\notin (pR)^D$.
Then the reduced block matrix $M$ has the following types:
\[
\operatorname{type}(M)=
\begin{cases}
(k-1,l+2), & \text{if } \alpha_1\neq 0,\\
(k,l+1), & \text{if } \alpha_1=0.
\end{cases}
\]
\end{lemma}

\begin{proof}
This follows from Lemma~\ref{lem:type-case-u-nonp-v-zero} by transposition.
Indeed, transposition preserves Smith type and interchanges
\[
u_3\leftrightarrow v_3,\qquad
\alpha_1\leftrightarrow \xi_1,\qquad
N\leftrightarrow D.
\]
Thus the assumptions
\[
u_3\notin(pR)^N,\qquad v_3=0
\]
in Lemma~\ref{lem:type-case-u-nonp-v-zero} correspond to
\[
u_3=0,\qquad v_3\notin(pR)^D
\]
here, and the condition $\xi_1\neq0$ corresponds to
$\alpha_1\neq0$. Applying
Lemma~\ref{lem:type-case-u-nonp-v-zero} to $M^t$ gives the stated
alternatives.
\end{proof}

\begin{lemma}
\label{lem:type-case-u-nonp-v-p-nonzero}
Assume that
$u_3\notin(pR)^N$ and $v_3\in(pR)^D\setminus\{0\}$.
Then the reduced block matrix $M$ has the following types:
\[
\operatorname{type}(M)=
\begin{cases}
(k-1,l+2), & \text{if } \xi_1\neq0,\\
(k+1,l+1), & \text{if } \xi_1=0.
\end{cases}
\]
\end{lemma}

\begin{proof}
Since $u_3\notin(pR)^N$, some entry of $u_3$ is a unit. By column
operations within the last $N$ columns, we may assume that
\[
u_3^t=(1,0,\ldots,0).
\]
The corresponding pivot column has zero entries outside the first row.
Hence, using this pivot column, we can eliminate all other entries in
the first row by column operations. In particular, $\rho'$ and
$\widetilde{\alpha}_1^t$ are eliminated. Thus one invariant factor equal to
$1$ is split off.

After this unit factor has been split off, the remaining contribution
from the first column is controlled by $\xi_1$ and $v_3$, while the
block $I_l$ still contributes $l$ invariant factors equal to $1$.

First suppose that $\xi_1\neq0$. Then a unit entry of $\widetilde{\xi}_1$ can
be used as a pivot. Since $v_3\in(pR)^D$, the entries of $v_3$ can
be eliminated using this pivot with coefficients in $pR$. The entries
created in the remaining blocks are multiples of $p^2$, and hence
vanish. Therefore the matrix is equivalent to the case $v_3=0$ with
$u_3\notin(pR)^N$ and $\xi_1\neq0$. By
Lemma~\ref{lem:type-case-u-nonp-v-zero}, the type is
\[
(k-1,l+2).
\]

Next suppose that $\xi_1=0$. Then no further unit factor is created
from the first column. Since $v_3\in(pR)^D\setminus\{0\}$, the vector
$v_3$ contributes one additional invariant factor equal to $p$.
The block $pI_k$ is otherwise unchanged after the unit factor coming
from $u_3$ has been split off. 
Therefore the matrix has $k+1$ invariant factors equal to $p$, and $l+1$ invariant factors equal to $1$.
Hence
\[
\operatorname{type}(M)=(k+1,l+1).
\]
\end{proof}

\begin{lemma}
\label{lem:type-case-u-p-nonzero-v-nonp}
Assume that
$u_3\in(pR)^N\setminus\{0\}$ and $v_3\notin(pR)^D$.
Then the reduced block matrix $M$ has the following types:
\[
\operatorname{type}(M)=
\begin{cases}
(k-1,l+2), & \text{if } \alpha_1\neq0,\\
(k+1,l+1), & \text{if } \alpha_1=0.
\end{cases}
\]
\end{lemma}

\begin{proof}
This follows from
Lemma~\ref{lem:type-case-u-nonp-v-p-nonzero} by transposition.
Indeed, transposition preserves Smith type and interchanges
\[
u_3\leftrightarrow v_3,\qquad
\alpha_1\leftrightarrow \xi_1,\qquad
N\leftrightarrow D.
\]
Thus the assumptions
\[
u_3\notin(pR)^N,
\qquad
v_3\in(pR)^D\setminus\{0\}
\]
in Lemma~\ref{lem:type-case-u-nonp-v-p-nonzero} correspond to
\[
u_3\in(pR)^N\setminus\{0\},
\qquad
v_3\notin(pR)^D
\]
here.
Under this correspondence, the condition $\xi_1\neq0$ becomes $\alpha_1\neq0$.
Applying Lemma~\ref{lem:type-case-u-nonp-v-p-nonzero} to $M^t$
gives the stated alternatives.
\end{proof}

\begin{lemma}
\label{lem:type-case-u-nonp-v-nonp}
Assume that
$u_3\notin(pR)^N$ and $v_3\notin(pR)^D$.
Then the reduced block matrix $M$ has type $(k,l+2)$.
\end{lemma}

\begin{proof}
Since $u_3\notin(pR)^N$, some entry of $u_3$ is a unit. By column
operations within the last $N$ columns, we may assume that
\[
u_3^t=(1,0,\ldots,0).
\]
The corresponding pivot column has zero entries outside the first row.
Hence, using this pivot column, we can eliminate all other entries in
the first row by column operations. In particular, $\rho'$ and
$\widetilde{\alpha}_1^t$ are eliminated. Thus one invariant factor equal to
$1$ is split off. These column operations do not affect the entries
below the first row.

Similarly, since $v_3\notin(pR)^D$, some entry of $v_3$ is a unit.
By row operations within the last $D$ rows, we may assume that
\[
v_3=(1,0,\ldots,0)^t.
\]
The corresponding pivot row has zero entries outside the first column.
Hence, using this pivot row, we can eliminate all other entries in the
first column by row operations. In particular, $\widetilde{\xi}_1$ is eliminated.
Thus a second invariant factor equal to $1$ is split off.

The two unit pivots come from the $u_3$- and $v_3$-blocks and do not
change the blocks $pI_k$ and $I_l$. Consequently, after splitting
off these two unit factors, the remaining nonzero part is equivalent
to
\[
pI_k\oplus I_l.
\]
Thus
\[
\operatorname{type}(M)=(k,l+2).
\]
\end{proof}

\begin{corollary}
\label{cor:possible-subtypes}
Let $M$ be the reduced block matrix
\[
M=
\begin{pmatrix}
\rho' & \widetilde{\alpha}_1^t & 0 & u_3^t\\
\widetilde{\xi}_1 & pI_k & 0 & 0\\
0 & 0 & I_l & 0\\
v_3 & 0 & 0 & 0
\end{pmatrix}.
\]
If $M$ has type $(i,j)$, then the type $(k,l)$ of the submatrix
$y$ belongs to
\[
\begin{aligned}
S_{i,j}=\{&
(i,j),\ (i-1,j),\ (i-2,j),\ (i,j-1),\ (i+1,j-1),\\
&(i-1,j-1),\ (i+2,j-2),\ (i+1,j-2),\ (i,j-2)\}.
\end{aligned}
\]
Here pairs outside $I_{d-1}$ are omitted. Moreover, for each listed
pair, the preceding lemmas specify the conditions under which that pair
contributes.
\end{corollary}

\begin{proof}
The preceding lemmas describe the type of $M$ in terms of the type
$(k,l)$ of the submatrix $y$. We now read those lists in the
opposite direction: fixing the resulting type $(i,j)$, we determine
which submatrix types $(k,l)$ can give rise to it.

Each of $u_3$ and $v_3$ is in exactly one of the following three
classes: it is zero, it is nonzero and contained in the corresponding
module $(pR)^N$ or $(pR)^D$, or it is not contained in that module.
Thus the preceding lemmas exhaust all possibilities for
$(u_3,v_3)$. Reading off the resulting types from
Lemmas~\ref{lem:type-case-u3-v3-zero}--\ref{lem:type-case-u-nonp-v-nonp}
gives the listed set $S_{i,j}$.
\end{proof}

\subsection{Enumeration of $N_{i,j}(\rho;k,l)$}
\label{subsec:enumeration-N}

Throughout this subsection, fix $(i,j)\in I_d$.
We now compute $N_{i,j}(\rho;k,l)$ for
$(k,l)\in S_{i,j}$. By Corollary~\ref{cor:possible-subtypes}, the
type conditions are expressed in terms of
\[
\rho',\qquad \alpha_1,\qquad \xi_1,\qquad u_3,\qquad v_3.
\]
It remains to translate the conditions on $\rho'$ into conditions on
the finite-field variables
\[
\alpha_1,\beta_1,\xi_1,\eta_1,
\alpha_2,\beta_2,\xi_2,\eta_2
\]
and then count the resulting systems.
Some formulas below contain apparent negative powers, such as $q^{i+2j-1}$.
In the boundary cases these expressions are interpreted after
simplification: the apparent negative powers either cancel with other
factors or occur together with zero factors.

Put 
\[
A = 
\widetilde{\beta}_1^t\widetilde{\xi}_1
+
\widetilde{\alpha}_1^t\widetilde{\eta}_1
+
\widetilde{\alpha}_2^t\widetilde{\xi}_2,
\qquad
B =  
\widetilde{\beta}_1^t\widetilde{\eta}_1
+
\widetilde{\beta}_2^t\widetilde{\xi}_2
+
\widetilde{\alpha}_2^t\widetilde{\eta}_2,
\]
so that $\rho' = \rho - A - pB$.
Reducing $A$ modulo $p$, we get
$\beta_1^t\xi_1+\alpha_1^t\eta_1+\alpha_2^t\xi_2$.
Thus, if
\[
\sigma=
\beta_1^t\xi_1+\alpha_1^t\eta_1+\alpha_2^t\xi_2\in\mathbb F_q,
\]
then $A-\widetilde{\sigma}\in pR$. Hence there exists a unique
$\kappa\in\mathbb F_q$ such that
\[
A=\widetilde{\sigma}+p\widetilde{\kappa}.
\]
Here $\kappa$ depends on
$(\beta_1,\xi_1,\alpha_1,\eta_1,\alpha_2,\xi_2)$, and records the $p$-component arising from the chosen representatives.
Similarly, 
$B\bmod p = \beta_1^t\eta_1+\beta_2^t\xi_2+\alpha_2^t\eta_2$.
Set
\[
\tau
=
\beta_1^t\eta_1+\beta_2^t\xi_2+\alpha_2^t\eta_2 \in \mathbb{F}_q.
\]
Since $B-\widetilde{\tau}\in pR$ and $p^2=0$, we have
\[
pB=p\widetilde{\tau}.
\]
Therefore
\[
A + pB = \widetilde{\sigma} + p\widetilde{\kappa} + p\widetilde{\tau} = \widetilde{\sigma} + p\left(\widetilde{\kappa+\tau}\right),
\]
where the last equality follows from $p^2=0$.
Indeed, $\widetilde{\kappa}+\widetilde{\tau}$ and
$\widetilde{\kappa+\tau}$ differ by an element of $pR$, which vanishes
after multiplication by $p$.

Although the dimensions of the variables depend on $(k,l)$, we use
the same notation $\sigma, \tau, \kappa$ throughout the following lemmas.

\begin{lemma}
\label{lem:carry-shift-count}
With the above notation, suppose $\alpha_1 = 0$ and $\xi_1 = 0$.
Then, for every $s,t\in\mathbb F_q$, the number of tuples
$(\beta_1,\eta_1,\alpha_2,\beta_2,\xi_2,\eta_2)$
satisfying
\[
\sigma=s,\qquad \kappa + \tau = t
\]
is equal to the number of such tuples satisfying
\[
\sigma=s,\qquad \tau=t.
\]
\end{lemma}
\begin{proof}
Since $\alpha_1=0$ and $\xi_1=0$, and since the representative of
$0\in\mathbb F_q$ is chosen to be $0\in R$, we have
\[
A=\widetilde{\alpha}_2^t\widetilde{\xi}_2, \qquad
\sigma = \alpha_2^t\xi_2.
\]

Fix $\alpha_2, \xi_2\in\mathbb F_q^l$ with $\alpha_2^t\xi_2=s$.
If $(\alpha_2, \xi_2)=(0,0)$, then necessarily $s=0$, and
$A=0$,
so $\kappa=0$. Hence the two conditions on $\tau$ are identical.
Now suppose $(\alpha_2, \xi_2)\neq (0,0)$. 
For fixed $\beta_1,\eta_1\in\mathbb F_q^k$,
we have
\[
\tau
=
\beta_1^t\eta_1+\beta_2^t\xi_2+\alpha_2^t\eta_2.
\]
As a function of
$(\beta_2,\eta_2)\in\mathbb F_q^l\times\mathbb F_q^l$,
the linear part
\[
(\beta_2,\eta_2)\longmapsto \beta_2^t\xi_2+\alpha_2^t\eta_2
\]
is nonzero, because $(\alpha_2, \xi_2)\neq (0,0)$. Therefore this map is a
surjective linear map onto $\mathbb F_q$. Hence, for every prescribed
value $t'\in\mathbb F_q$, the number of pairs
$(\beta_2,\eta_2)$
satisfying
\[
\beta_2^t\xi_2+\alpha_2^t\eta_2=t'
\]
is the same, namely $q^{2l-1}$.
Consequently, for fixed $\alpha_2,\xi_2,\beta_1,\eta_1$,
the number of choices of $(\beta_2,\eta_2)$ satisfying
$\tau=t-\kappa$
is equal to the number of choices satisfying
$\tau=t$.
Finally, $\beta_1,\eta_1$ are arbitrary, and the above comparison
holds for every fixed pair $\alpha_2, \xi_2\in\mathbb F_q^l$ with $\alpha_2^t\xi_2=s$.
Summing over all
such $(\alpha_2, \xi_2)$, and over all $\beta_1,\eta_1$, gives the desired
equality of cardinalities.
\end{proof}

In each lemma below, we use the type criteria from
Subsection~\ref{subsec:type-criteria}, translate the conditions on
$\rho'$ into conditions on $\sigma$ and $\kappa+\tau$, and then use
Lemma~\ref{lem:carry-shift-count}, when applicable, to replace
$\kappa+\tau$ by $\tau$ for counting purposes.
The elementary finite-field counts used below are collected in Appendix~A.

\begin{lemma}
\label{lem:N-ij-ij}
Assume that $(i,j)\in I_{d-1}$. Then
\begin{align}
N_{i,j}(0;i,j)
&=
q^{i+2j-1}
\left(
q^{i+2j-1}+q^{i+j}-q^{i+j-1}+q-1
\right), \label{eq:N-ij-ij-0}\\[4pt]
N_{i,j}(p;i,j)
&=
q^{i+2j-1}
\left(
q^{i+2j-1}+q^{i+j}-q^{i+j-1}-1
\right), \label{eq:N-ij-ij-p}\\[4pt]
N_{i,j}(1;i,j)
&=
q^{2i+3j-2}(q^j-1). \label{eq:N-ij-ij-1}
\end{align}
\end{lemma}

\begin{proof}
In the case $(k,l)=(i,j)$, the type criteria in
Subsection~\ref{subsec:type-criteria} show that the reduced block
matrix has type $(i,j)$ if and only if
\[
\rho'=0,\qquad
\alpha_1=0,\qquad
\xi_1=0,\qquad
u_3=0,\qquad
v_3=0.
\]
Under the conditions
$\alpha_1=\xi_1=0$, we have
\[
\sigma=\alpha_2^t\xi_2,
\qquad
\tau=
\beta_1^t\eta_1+\beta_2^t\xi_2+\alpha_2^t\eta_2.
\]
If $\rho=0$, then $\rho'=0$ is equivalent to
$\sigma=0$ and $\kappa+\tau=0$.
By Lemma~\ref{lem:carry-shift-count}, the number of such tuples is
equal to the number of tuples satisfying
$\sigma=0$ and $\tau=0$.
By Lemma~\ref{lem:sixtuple-zero-constrained-count}, applied with $k=i$ and $l=j$, this gives \eqref{eq:N-ij-ij-0}.

If $\rho=p$, then $\rho'=0$ is equivalent to
$\sigma=0$ and $\kappa+\tau=1$.
By Lemma~\ref{lem:carry-shift-count}, the number of such tuples is
equal to the number of tuples satisfying
$\sigma=0$ and $\tau=1$.
By Lemma~\ref{lem:sixtuple-zero-cross-fixed-nonzero-sum-count}, applied
with $k=i$ and $l=j$, this gives
\eqref{eq:N-ij-ij-p}.

If $\rho=1$, then $\rho'=0$ is equivalent to
$\sigma=1$ and $\kappa+\tau=0$.
By Lemma~\ref{lem:carry-shift-count}, the number of such tuples is
equal to the number of tuples satisfying
$\sigma=1$ and $\tau=0$.
By Lemma~\ref{lem:sixtuple-nonzero-cross-zero-total-count}, applied
with $k=i$ and $l=j$, this gives
\eqref{eq:N-ij-ij-1}.
\end{proof}

\begin{lemma}
\label{lem:N-ij-i-1-j}
Assume that $(i-1,j)\in I_{d-1}$. Then
\begin{align}
N_{i,j}(0;i-1,j)
&=
q^{2i+3j-4}
\left(q^j+q-1\right)
\left(q^{n-i-j+1}+q^{d-i-j+1}-q-1\right)
- q^{i+2j-2}(q-1), \label{eq:N-i-1-j-0}\\[4pt]
N_{i,j}(p;i-1,j)
&=
q^{2i+3j-4}
\left(q^j+q-1\right)
\left(q^{n-i-j+1}+q^{d-i-j+1}-q-1\right)
+ q^{i+2j-2}, \label{eq:N-i-1-j-p}\\[4pt]
N_{i,j}(1;i-1,j)
&=
q^{2i+3j-4}
\left(q^j-1\right)
\left(q^{n-i-j+1}+q^{d-i-j+1}-q-1\right).
\label{eq:N-i-1-j-1}
\end{align}
\end{lemma}

\begin{proof}
Put $N=n-i-j$ and $D=d-i-j$.
We consider the case $(k,l)=(i-1,j)$. By the type criteria in
Subsection~\ref{subsec:type-criteria}, the reduced block matrix has
type $(i,j)$ if and only if 
\[
\alpha_1=0,\qquad \xi_1=0,
\]
and one of the following two disjoint conditions holds:
\[
\begin{array}{ll}
\textup{(a)}&
\rho'=0,\qquad
(u_3,v_3)\in(pR)^N\times(pR)^D,
\qquad
\text{exactly one of $u_3$ and $v_3$ is zero},\\[2mm]
\textup{(b)}&
\rho'\in pR\setminus\{0\},\qquad
(u_3,v_3)\in(pR)^N\times(pR)^D,
\qquad
u_3=0\text{ or }v_3=0.
\end{array}
\]
Under the conditions $\alpha_1=\xi_1=0$, we have
\[
\sigma =\alpha_2^t\xi_2,
\qquad
\tau =\beta_1^t\eta_1+\beta_2^t\xi_2+\alpha_2^t\eta_2.
\]

First suppose that $\rho=0$. Then
\[
\rho'=0
\quad\Longleftrightarrow\quad
\sigma=0,\ \kappa+\tau=0,
\qquad
\rho'\in pR\setminus\{0\}
\quad\Longleftrightarrow\quad
\sigma=0,\ \kappa+\tau\neq0.
\]
By Lemma~\ref{lem:carry-shift-count}, the number of tuples satisfying
$\sigma=0$ and $\kappa+\tau=0$ is equal to the number of tuples satisfying
$\sigma=0$ and $\tau=0$; similarly, the number of tuples satisfying
$\sigma=0$ and $\kappa+\tau\neq0$ is equal to the number of tuples satisfying
$\sigma=0$ and $\tau\neq0$.

For the case (a), by Lemma~\ref{lem:sixtuple-zero-constrained-count}, applied with
$k=i-1$ and $l=j$, the number of choices satisfying $\sigma=0$ and $\tau=0$
is
\[
q^{i+2j-2}
\left(
q^{i+2j-2}+q^{i+j-1}-q^{i+j-2}+q-1
\right).
\]
For this contribution, exactly one of $u_3=0$ and $v_3=0$ must
hold. Since $u_3\in(pR)^N$ and $v_3\in(pR)^D$, the number of such
pairs is $q^N+q^D-2$.
Hence the first contribution is
\[
q^{i+2j-2}
\left(
q^{i+2j-2}+q^{i+j-1}-q^{i+j-2}+q-1
\right)
\left(q^N+q^D-2\right).
\]
For the case (b), by Lemma~\ref{lem:sixtuple-zero-cross-nonzero-sum-count},
applied with $k=i-1$ and $l=j$, the number of choices satisfying $\sigma=0$ and $\tau \neq 0$
is
\[
q^{i+2j-2}(q-1)
\left(
q^{i+2j-2}+q^{i+j-1}-q^{i+j-2}-1
\right).
\]
In case (b), at least one of $u_3$ and $v_3$ must be zero, giving $q^N+q^D-1$ choices.
Hence the second contribution is
\[
q^{i+2j-2}(q-1)
\left(
q^{i+2j-2}+q^{i+j-1}-q^{i+j-2}-1
\right)
\left(q^N+q^D-1\right).
\]
Adding the two contributions and simplifying, we obtain
\[
N_{i,j}(0;i-1,j)
=
q^{2i+3j-4}
\left(q^j+q-1\right)
\left(q^{N+1}+q^{D+1}-q-1\right)
- q^{i+2j-2}(q-1).
\]
Substituting $N=n-i-j$ and $D=d-i-j$ gives \eqref{eq:N-i-1-j-0}.

Next suppose that $\rho=p$. Then
\[
\rho'=0
\quad\Longleftrightarrow\quad
\sigma=0,\ \kappa+\tau=1,
\qquad
\rho'\in pR\setminus\{0\}
\quad\Longleftrightarrow\quad
\sigma=0,\ \kappa+\tau\neq 1.
\]
By Lemma~\ref{lem:carry-shift-count}, the number of tuples satisfying
$\sigma=0$ and $\kappa+\tau=1$ is equal to the number of tuples satisfying
$\sigma=0$ and $\tau=1$; similarly, the number of tuples satisfying
$\sigma=0$ and $\kappa+\tau\neq 1$ is equal to the number of tuples satisfying
$\sigma=0$ and $\tau\neq 1$.

For the case (a), by Lemma~\ref{lem:sixtuple-zero-cross-fixed-nonzero-sum-count},
applied with $k=i-1$ and $l=j$, the number of choices satisfying $\sigma=0$ and $\tau=1$
is
\[
q^{i+2j-2}
\left(
q^{i+2j-2}+q^{i+j-1}-q^{i+j-2}-1
\right).
\]
For this contribution, exactly one of $u_3=0$ and $v_3=0$ must
hold, giving $q^N+q^D-2$
choices for $(u_3,v_3)$. Hence the first contribution is
\[
q^{i+2j-2}
\left(
q^{i+2j-2}+q^{i+j-1}-q^{i+j-2}-1
\right)
\left(q^N+q^D-2\right).
\]
For the case (b), by Lemma~\ref{lem:sixtuple-zero-cross-not-fixed-nonzero-sum-count},
applied with $k=i-1$ and $l=j$, the number of choices satisfying $\sigma=0$ and $\tau \neq 1$
is
\[
q^{i+2j-2}
\left(
q^{i+2j-1}-q^{i+2j-2}
+q^{i+j}-2q^{i+j-1}+q^{i+j-2}+1
\right).
\]
For this contribution, we require $u_3 = 0$ or $v_3 = 0$,
which gives $q^N+q^D-1$
choices for $(u_3,v_3)$. Hence the second contribution is
\[
q^{i+2j-2}
\left(
q^{i+2j-1}-q^{i+2j-2}
+q^{i+j}-2q^{i+j-1}+q^{i+j-2}+1
\right)
\left(q^N+q^D-1\right).
\]
Adding the two contributions and simplifying, we obtain
\[
N_{i,j}(p;i-1,j)
=
q^{2i+3j-4}
\left(q^j+q-1\right)
\left(q^{N+1}+q^{D+1}-q-1\right)
+q^{i+2j-2}.
\]
Substituting $N=n-i-j$ and $D=d-i-j$ gives \eqref{eq:N-i-1-j-p}.

Finally suppose that $\rho=1$. Then
\[
\rho'=0
\quad\Longleftrightarrow\quad
\sigma=1,\ \kappa+\tau=0,
\qquad
\rho'\in pR\setminus\{0\}
\quad\Longleftrightarrow\quad
\sigma=1,\ \kappa+\tau\neq 0.
\]
By Lemma~\ref{lem:carry-shift-count}, the number of tuples satisfying
$\sigma=1$ and $\kappa+\tau=0$ is equal to the number of tuples satisfying
$\sigma=1$ and $\tau=0$; similarly, the number of tuples satisfying
$\sigma=1$ and $\kappa+\tau\neq 0$ is equal to the number of tuples satisfying
$\sigma=1$ and $\tau\neq 0$.

For the case (a), by Lemma~\ref{lem:sixtuple-nonzero-cross-zero-total-count}, applied
with $k=i-1$ and $l=j$, the number of choices satisfying $\sigma=1$ and $\tau=0$
is
\[
q^{2i+3j-4}(q^j-1).
\]
For this contribution, exactly one of $u_3=0$ and $v_3=0$ must
hold, giving $q^N+q^D-2$
choices for $(u_3,v_3)$. Hence the first contribution is
\[
q^{2i+3j-4}(q^j-1)(q^N+q^D-2).
\]
For the case (b), by Lemma~\ref{lem:sixtuple-nonzero-cross-nonzero-sum-count},
applied with $k=i-1$ and $l=j$, the number of choices satisfying $\sigma=1$ and $\tau\neq 0$
is
\[
q^{2i+3j-4}(q-1)(q^j-1).
\]
For this contribution, we require $u_3=0$ or $v_3=0$,
which gives $q^N+q^D-1$
choices for $(u_3,v_3)$. Hence the second contribution is
\[
q^{2i+3j-4}(q-1)(q^j-1)(q^N+q^D-1).
\]
Adding the two contributions and simplifying, we obtain
\[
N_{i,j}(1;i-1,j)
=
q^{2i+3j-4}(q^j-1)
\left(q^{N+1}+q^{D+1}-q-1\right).
\]
Substituting $N=n-i-j$ and $D=d-i-j$ gives \eqref{eq:N-i-1-j-1}.
\end{proof}

\begin{lemma}
\label{lem:N-ij-i-2-j}
Assume that $(i-2,j)\in I_{d-1}$. Then
\begin{align}
N_{i,j}(\rho;i-2,j)
&=
q^{2i+3j-5}
(q^j+q-1)
(q^{n-i-j+1}-1)
(q^{d-i-j+1}-1),
\qquad \rho\in\{0,p\}, \label{eq:N-i-2-j-0p}\\[4pt]
N_{i,j}(1;i-2,j)
&=
q^{2i+3j-5}
(q^j-1)
(q^{n-i-j+1}-1)
(q^{d-i-j+1}-1).
\label{eq:N-i-2-j-1}
\end{align}
\end{lemma}

\begin{proof}
Put
$N=n-i-j+1$ and $D=d-i-j+1$.
We consider the case $(k,l)=(i-2,j)$. By the type criteria in
Subsection~\ref{subsec:type-criteria}, the reduced block matrix has
type $(i,j)$ if and only if
\[
\rho'\in pR,\qquad
\alpha_1=0,\qquad
\xi_1=0,\qquad
u_3\in(pR)^N\setminus\{0\},
\qquad
v_3\in(pR)^D\setminus\{0\}.
\]
Under the conditions $\alpha_1=\xi_1=0$, we have
\[
\sigma=\alpha_2^t\xi_2.
\]

First let $\rho\in\{0,p\}$.
Since both values reduce to $0$ modulo $p$, the condition $\rho'\in pR$ is equivalent to
$\sigma=0$.
The variables
$\beta_1,\eta_1\in\mathbb{F}_q^{i-2}$,
$\beta_2,\eta_2\in\mathbb{F}_q^j$
are unrestricted.
By Lemma~\ref{lem:zero-inner-product-count}, applied in dimension $j$,
the number of choices satisfying $\sigma=0$ is
\[
q^{2i+2j-4}\cdot q^{j-1}(q^j+q-1)
=
q^{2i+3j-5}(q^j+q-1).
\]
There are $(q^N-1)(q^D-1)$ choices for $(u_3,v_3)$.
Thus
\[
N_{i,j}(\rho;i-2,j)
=
q^{2i+3j-5}(q^j+q-1)(q^N-1)(q^D-1),
\qquad \rho\in\{0,p\}.
\]
Substituting $N=n-i-j+1$, $D=d-i-j+1$
gives \eqref{eq:N-i-2-j-0p}.

Next suppose that $\rho=1$. Since $\rho$ reduces to $1$ modulo $p$,
the condition $\rho'\in pR$ is equivalent to $\sigma=1$.
Again, the variables
$\beta_1,\eta_1,\beta_2,\eta_2$
are unrestricted.
By Lemma~\ref{lem:nonzero-inner-product-count}, applied in dimension $j$, the
number of choices of the finite-field variables satisfying $\sigma=1$
is
\[
q^{2i+2j-4}\cdot q^{j-1}(q^j-1)
=
q^{2i+3j-5}(q^j-1).
\]
The choices of $u_3$ and $v_3$ are again $(q^N-1)(q^D-1)$.
Therefore
\[
N_{i,j}(1;i-2,j)
=
q^{2i+3j-5}(q^j-1)(q^N-1)(q^D-1),
\]
which gives \eqref{eq:N-i-2-j-1}.
\end{proof}

\begin{lemma}
\label{lem:N-ij-i-j-1}
Assume that $(i,j-1)\in I_{d-1}$. Then
\begin{align}
N_{i,j}(\rho;i,j-1)
&=
q^{2i+4j-5}(q^i-1)(2q^{n+d-2i-2j} - q^{n-i-j} - q^{d-i-j}) \notag\\
&\quad
+ q^{3i+4j-4}\left( q^{n-i-j}(q^{n-i-j} - 1) + q^{d-i-j}(q^{d-i-j}-1)\right)\notag\\
&\quad
+ q^{n+d+j-4}(q-1)(2q^{i+j-1}-q^{j-1}-1),
\qquad \rho\in\{0,p\}, \label{eq:N-i-j-1-0p}\\[4pt]
N_{i,j}(1;i,j-1)
&=
q^{2i+4j-5}(q^i-1)(2q^{n+d-2i-2j} - q^{n-i-j} - q^{d-i-j}) \notag\\
&\quad
+ q^{3i+4j-4}\left( q^{n-i-j}(q^{n-i-j} - 1) + q^{d-i-j}(q^{d-i-j}-1)\right)\notag\\
&\quad
+ q^{n+d+j-4}(2q^{i+j} - 2q^{i+j-1}-q^j+q^{j-1}+1).
\label{eq:N-i-j-1-1}
\end{align}
\end{lemma}

\begin{proof}
Put
$N=n-i-j$ and $D=d-i-j$.
We consider the case $(k,l)=(i,j-1)$. By the type criteria in
Subsection~\ref{subsec:type-criteria}, the reduced block matrix has
type $(i,j)$ if and only if one of the following five mutually
disjoint conditions holds:
\[
\begin{array}{ll}
\textup{(a)}&
\rho'\notin pR,\quad
\alpha_1=0\text{ or }\xi_1=0,\quad
u_3\in(pR)^N,\quad v_3\in(pR)^D,\\[2mm]
\textup{(b)}&
\rho'\in pR,\quad
\alpha_1=0,\quad \xi_1\neq0,\quad
u_3\in(pR)^N\setminus\{0\},\quad v_3\in(pR)^D,\\[2mm]
\textup{(c)}&
\rho'\in pR,\quad
\alpha_1\neq0,\quad \xi_1=0,\quad
u_3\in(pR)^N,\quad v_3\in(pR)^D\setminus\{0\},\\[2mm]
\textup{(d)}&
\alpha_1=0,\quad \xi_1\ \text{arbitrary},\quad
u_3=0,\quad v_3\notin(pR)^D,\\[2mm]
\textup{(e)}&
\alpha_1\ \text{arbitrary},\quad \xi_1=0,\quad
u_3\notin(pR)^N,\quad v_3=0.
\end{array}
\]
We count the contributions of \textup{(a)}--\textup{(e)} separately.

First let $\rho\in\{0,p\}$. Since both values reduce to $0$ modulo $p$,
\[
\rho'\in pR
\quad\Longleftrightarrow\quad
\sigma =0,
\qquad
\rho'\notin pR
\quad\Longleftrightarrow\quad
\sigma \neq0.
\]
For case \textup{(a)}, the variables
$\beta_2,\eta_2$
are unrestricted. By
Lemma~\ref{lem:sixtuple-nonzero-sum-xi-one-or-eta-two-zero-count},
applied with $k=i$ and $l=j-1$, the number of choices of the finite-field variables satisfying
\[
\sigma =\beta_1^t\xi_1+\alpha_1^t\eta_1+\alpha_2^t\xi_2\neq0
\]
and such that at least one of $\alpha_1$ and $\xi_1$ is zero is
\[
q^{2j-2} \cdot q^{2i+j-2}(q-1)
\left(2q^{i+j-1}-q^{j-1}-1\right)
=
q^{2i+3j-4}(q-1)
\left(2q^{i+j-1}-q^{j-1}-1\right).
\]
For this contribution, we require $u_3\in(pR)^N$ and $v_3\in(pR)^D$,
which gives $q^{N+D}$
choices for $(u_3,v_3)$. Hence the first
contribution is
\[
q^{N+D+2i+3j-4}(q-1)
\left(2q^{i+j-1}-q^{j-1}-1\right).
\]
For case \textup{(b)}, the variables
$\beta_2,\eta_1,\eta_2$
are unrestricted. By
Lemma~\ref{lem:two-block-fixed-sum-eta-one-nonzero-count}, applied
with $k=i$ and $l=j-1$, the number of choices of the finite-field variables
satisfying
\[
\sigma=\beta_1^t\xi_1+\alpha_2^t\xi_2=0,
\qquad
\xi_1\neq0
\]
is
\[
q^{i+2j-2}\cdot q^{i+2j-3}(q^i-1)
=
q^{2i+4j-5}(q^i-1).
\]
For this contribution, we require $u_3\in(pR)^N\setminus\{0\}$ and $v_3\in(pR)^D$,
which gives $q^D(q^N-1)$
choices for $(u_3,v_3)$. Hence the second
contribution is
\[
q^{D+2i+4j-5}(q^i-1)(q^N-1).
\]
For case \textup{(c)}, the variables
$\beta_1,\beta_2,\eta_2$
are unrestricted. By
Lemma~\ref{lem:two-block-fixed-sum-xi-one-nonzero-count}, applied with $k=i$ and $l=j-1$, the number of choices
satisfying
\[
\sigma=\alpha_1^t\eta_1+\alpha_2^t\xi_2=0,
\qquad
\alpha_1\neq0
\]
is
\[
q^{i+2j-2}\cdot q^{i+2j-3}(q^i-1)
=
q^{2i+4j-5}(q^i-1).
\]
For this contribution, we require $u_3\in(pR)^N$ and $v_3\in(pR)^D\setminus\{0\}$,
which gives $q^N(q^D-1)$
choices for $(u_3,v_3)$. Hence the third
contribution is
\[
q^{N+2i+4j-5}(q^i-1)(q^D-1).
\]
For case \textup{(d)}, the variables
$\beta_1,\alpha_2,\beta_2,\xi_1,\eta_1,\xi_2,\eta_2$
are unrestricted, giving $q^{3i+4j-4}$ choices.
Since $u_3=0$ and $v_3\notin(pR)^D$, there are $q^D(q^D-1)$
choices for $(u_3,v_3)$. Hence the fourth contribution is
\[
q^{D+3i+4j-4}(q^D-1).
\]
For case \textup{(e)}, the variables
$\alpha_1,\beta_1,\alpha_2,\beta_2,\eta_1,\xi_2,\eta_2$
are unrestricted, giving $q^{3i+4j-4}$
choices. Since $u_3\notin(pR)^N$ and $v_3=0$, there are $q^N(q^N-1)$
choices for $(u_3,v_3)$. Hence the fifth contribution is
\[
q^{N+3i+4j-4}(q^N-1).
\]
Adding the five contributions and simplifying, we obtain
\[
\begin{aligned}
N_{i,j}(\rho;i,j-1)
&=
q^{2i+4j-5}(q^i-1)(2q^{N+D} - q^N - q^D) \\
&\quad
+ q^{3i+4j-4}\left( q^N(q^N - 1) + q^D(q^D-1)\right)\\
&\quad
+ q^{N+D+2i+3j-4}(q-1)(2q^{i+j-1}-q^{j-1}-1)
\end{aligned}
\]
for $\rho\in\{0,p\}$.
Substituting $N=n-i-j$ and $D=d-i-j$ gives \eqref{eq:N-i-j-1-0p}.

Next suppose that $\rho=1$. Then
\[
\rho'\in pR
\quad\Longleftrightarrow\quad
\sigma=1,
\qquad
\rho'\notin pR
\quad\Longleftrightarrow\quad
\sigma\neq1.
\]
For case \textup{(a)}, the variables
$\beta_2,\eta_2$
are unrestricted. By
Lemma~\ref{lem:sixtuple-not-fixed-sum-xi-one-or-eta-two-zero-count},
applied with $k=i$ and $l=j-1$, the number
of choices of the finite-field variables satisfying
\[
\sigma=\beta_1^t\xi_1+\alpha_1^t\eta_1+\alpha_2^t\xi_2\neq1
\]
and such that at least one of $\alpha_1$ and $\xi_1$ is zero is
\[
q^{2j-2} \cdot
q^{2i+j-2}
\left(
2q^{i+j}-2q^{i+j-1}-q^j+q^{j-1}+1
\right)
=
q^{2i+3j-4}
\left(
2q^{i+j}-2q^{i+j-1}-q^j+q^{j-1}+1
\right).
\]
For this contribution, we require $u_3\in(pR)^N$ and $v_3\in(pR)^D$,
which gives 
$q^{N+D}$
choices for $(u_3,v_3)$. Hence the first
contribution is
\[
q^{N+D+2i+3j-4}
\left(
2q^{i+j}-2q^{i+j-1}-q^j+q^{j-1}+1
\right).
\]
For case \textup{(b)}, the variables
$\beta_2,\eta_1,\eta_2$
are unrestricted. By
Lemma~\ref{lem:two-block-fixed-sum-eta-one-nonzero-count}, applied
with $k=i$ and $l=j-1$, the number of choices of the finite-field variables
satisfying
\[
\sigma=\beta_1^t\xi_1+\alpha_2^t\xi_2=1,
\qquad
\xi_1\neq0
\]
is
\[
q^{i+2j-2}\cdot q^{i+2j-3}(q^i-1)
=
q^{2i+4j-5}(q^i-1).
\]
For this contribution, we require $u_3\in(pR)^N\setminus\{0\}$ and $v_3\in(pR)^D$,
which gives 
$q^D(q^N-1)$
choices for $(u_3,v_3)$. Hence the second
contribution is
\[
q^{D+2i+4j-5}(q^i-1)(q^N-1).
\]
For case \textup{(c)}, the variables
$\beta_1,\beta_2,\eta_2$
are unrestricted. By
Lemma~\ref{lem:two-block-fixed-sum-xi-one-nonzero-count}, applied with $k=i$ and $l=j-1$, the number of choices of the finite-field variables
satisfying
\[
\sigma=\alpha_1^t\eta_1+\alpha_2^t\xi_2=1,
\qquad
\alpha_1\neq0
\]
is
\[
q^{i+2j-2}\cdot q^{i+2j-3}(q^i-1)
=
q^{2i+4j-5}(q^i-1).
\]
For this contribution, we require $u_3\in(pR)^N$ and $v_3\in(pR)^D\setminus\{0\}$,
which gives 
$q^N(q^D-1)$
choices for $(u_3,v_3)$. Hence the third
contribution is
\[
q^{N+2i+4j-5}(q^i-1)(q^D-1).
\]
For case \textup{(d)}, the variables
$\beta_1,\alpha_2,\beta_2,\xi_1,\eta_1,\xi_2,\eta_2$
are unrestricted, giving $q^{3i+4j-4}$
choices. Since $u_3=0$ and $v_3\notin(pR)^D$, there are
$q^D(q^D-1)$
choices for $(u_3,v_3)$. Hence the fourth contribution is
\[
q^{D+3i+4j-4}(q^D-1).
\]
For case \textup{(e)}, the variables
$\alpha_1,\beta_1,\alpha_2,\beta_2,\eta_1,\xi_2,\eta_2$
are unrestricted, giving $q^{3i+4j-4}$
choices. Since $u_3\notin(pR)^N$ and $v_3=0$, there are
$q^N(q^N-1)$
choices for $(u_3,v_3)$. Hence the fifth contribution is
\[
q^{N+3i+4j-4}(q^N-1).
\]
Adding the five contributions and simplifying, we obtain
\[
\begin{aligned}
N_{i,j}(1;i,j-1)
&=
q^{2i+4j-5}(q^i-1)(2q^{N+D} - q^N - q^D) \\
&\quad
+ q^{3i+4j-4}\left( q^N(q^N - 1) + q^D(q^D-1)\right)\\
&\quad
+ q^{N+D+2i+3j-4}(2q^{i+j} - 2q^{i+j-1}-q^j+q^{j-1}+1).
\end{aligned}
\]
Substituting $N=n-i-j$ and $D=d-i-j$ gives \eqref{eq:N-i-j-1-1}.
\end{proof}

\begin{lemma}
\label{lem:N-ij-i+1-j-1}
Assume that $(i+1,j-1)\in I_{d-1}$. Then
\[
N_{i,j}(\rho;i+1,j-1)
=
q^{2i+4j-3}
(q^{i+1}-1)
(q^{n-i-j-1}+q^{d-i-j-1}),
\qquad \rho\in\{0,p,1\}.
\]
\end{lemma}

\begin{proof}
Put
$N=n-i-j-1$ and $D=d-i-j-1$.
We consider the case $(k,l)=(i+1,j-1)$. By the type criteria in
Subsection~\ref{subsec:type-criteria}, the reduced block matrix has
type $(i,j)$ precisely in the following two cases:
\[
\begin{array}{ll}
\textup{(a)}&
\rho'\in pR,\qquad
\alpha_1=0,\qquad \xi_1\neq0,\qquad
u_3=0, \qquad
v_3\in(pR)^D,\\[2mm]
\textup{(b)}&
\rho'\in pR,\qquad
\alpha_1\neq0,\qquad \xi_1=0,\qquad
u_3\in(pR)^N, \qquad
v_3=0.
\end{array}
\]

The condition $\rho'\in pR$ is equivalent to
\[
\sigma =c_\rho,
\qquad
c_\rho=
\begin{cases}
0, & \rho\in\{0,p\},\\
1, & \rho=1.
\end{cases}
\]
For case \textup{(a)}, the variables
$\beta_2,\eta_1,\eta_2$
are unrestricted. By
Lemma~\ref{lem:two-block-fixed-sum-eta-one-nonzero-count}, applied
with $k=i+1$ and $l=j-1$, the number of
choices satisfying
\[
\sigma=\beta_1^t\xi_1+\alpha_2^t\xi_2=c_\rho,
\qquad
\xi_1\neq0
\]
is
\[
q^{i+2j-1}\cdot q^{i+2j-2}(q^{i+1}-1)
=
q^{2i+4j-3}(q^{i+1}-1).
\]
There are $q^D$ choices for $(u_3,v_3)$ in case \textup{(a)}, so this case contributes
\[
q^{D+2i+4j-3}(q^{i+1}-1).
\]
For case \textup{(b)}, the variables
$\beta_1,\beta_2,\eta_2$
are unrestricted. By
Lemma~\ref{lem:two-block-fixed-sum-xi-one-nonzero-count}, applied with $k=i+1$ and $l=j-1$, the number of choices satisfying
\[
\sigma = \alpha_1^t\eta_1+\alpha_2^t\xi_2=c_\rho,
\qquad
\alpha_1\neq0
\]
is
\[
q^{i+2j-1}\cdot q^{i+2j-2}(q^{i+1}-1)
=
q^{2i+4j-3}(q^{i+1}-1).
\]
For this contribution, we require $u_3\in(pR)^N$ and $v_3=0$,
which gives 
$q^N$
choices for $(u_3,v_3)$. Hence the second
contribution is
\[
q^{N+2i+4j-3}(q^{i+1}-1).
\]

Adding the two contributions, we obtain
\[
N_{i,j}(\rho;i+1,j-1)
=
q^{2i+4j-3}(q^{i+1}-1)(q^N+q^D)
\]
for all $\rho\in\{0,p,1\}$.
Substituting $N=n-i-j-1$ and $D=d-i-j-1$ gives the stated formula.
\end{proof}

\begin{lemma}
\label{lem:N-ij-i-1-j-1}
Assume that $(i-1,j-1)\in I_{d-1}$. Then
\[
N_{i,j}(\rho;i-1,j-1)
=
q^{3i+4j-7}
(q^{n-i-j+1}+q^{d-i-j+1})
(q^{n-i-j+1}-1)
(q^{d-i-j+1}-1),
\qquad \rho\in\{0,p,1\}.
\]
\end{lemma}

\begin{proof}
Put
$N=n-i-j+1$ and $D=d-i-j+1$.
We consider the case $(k,l)=(i-1,j-1)$. By the type criteria in
Subsection~\ref{subsec:type-criteria},
the reduced block matrix has type $(i,j)$ precisely in the following two cases,
which are transposes of each other:
\[
\begin{array}{ll}
\textup{(a)}&
\alpha_1=0,\qquad
u_3\in(pR)^N\setminus\{0\},\qquad
v_3\notin(pR)^D,\\[2mm]
\textup{(b)}&
\xi_1=0,\qquad
u_3\notin(pR)^N,\qquad
v_3\in(pR)^D\setminus\{0\}.
\end{array}
\]
In both cases, no condition is imposed on $\rho'$. Hence the count is
independent of $\rho\in\{0,p,1\}$.

For case \textup{(a)}, the variables
$\beta_1,\alpha_2,\beta_2,\xi_1,\eta_1,\xi_2,\eta_2$
are unrestricted, giving $q^{3i+4j-7}$
choices. Since $u_3\in(pR)^N\setminus\{0\}$ and $v_3\notin(pR)^D$, there are
$q^D(q^N-1)(q^D-1)$
choices for $(u_3,v_3)$. Hence the first contribution is
\[
q^{D+3i+4j-7}
(q^N-1)
(q^D-1).
\]

For case \textup{(b)}, the variables
$\alpha_1,\beta_1,\alpha_2,\beta_2,\eta_1,\xi_2,\eta_2$
are unrestricted, giving $q^{3i+4j-7}$
choices. Since $u_3\notin(pR)^N$ and $v_3\in(pR)^D\setminus\{0\}$, there are
$q^N(q^N-1)(q^D-1)$
choices for $(u_3,v_3)$. Hence the second contribution is
\[
q^{N+3i+4j-7}
(q^N-1)
(q^D-1).
\]

Adding the two contributions, we obtain
\[
N_{i,j}(\rho;i-1,j-1)
=
q^{3i+4j-7}
(q^N+q^D)
(q^N-1)
(q^D-1)
\]
for all $\rho\in\{0,p,1\}$.
Substituting $N=n-i-j+1$, $D=d-i-j+1$ gives the stated formula.
\end{proof}

\begin{lemma}
\label{lem:N-ij-i+2-j-2}
Assume that $(i+2,j-2)\in I_{d-1}$. Then
\[
N_{i,j}(\rho;i+2,j-2)
=
q^{n+d+2j-6}
(q^{i+2}-1)
(q^{i+1}-1),
\qquad \rho\in\{0,p,1\}.
\]
\end{lemma}

\begin{proof}
Put
$N=n-i-j-1$ and $D=d-i-j-1$.
We consider the case $(k,l)=(i+2,j-2)$. By the type criteria in
Subsection~\ref{subsec:type-criteria},
the reduced block matrix has type $(i,j)$ if and only if
\[
\alpha_1\neq0,\qquad
\xi_1\neq0,\qquad
\alpha_1^t\xi_1=0,\qquad
u_3\in(pR)^N,\qquad
v_3\in(pR)^D.
\]
No condition is imposed on $\rho'$, and hence the count is independent
of $\rho\in\{0,p,1\}$.

The variables
$\beta_1,\alpha_2,\beta_2,\eta_1,\xi_2,\eta_2$
are unrestricted.
By Lemma~\ref{lem:nonzero-vectors-zero-inner-product-count},
applied in dimension $i+2$, the total number of choices satisfying
\[
\alpha_1\neq0,\qquad
\xi_1\neq0,\qquad
\alpha_1^t\xi_1=0
\]
is
\[
q^{2i+4j-4}(q^{i+2}-1)(q^{i+1}-1).
\]
There are $q^{N+D}$ choices for $(u_3,v_3)$. Hence
\[
N_{i,j}(\rho;i+2,j-2)
=
q^{N+D+2i+4j-4}
(q^{i+2}-1)
(q^{i+1}-1)
\]
for all $\rho\in\{0,p,1\}$.
Substituting $N=n-i-j-1$ and $D=d-i-j-1$ gives the stated formula.
\end{proof}

\begin{lemma}
\label{lem:N-ij-i+1-j-2}
Assume that $(i+1,j-2)\in I_{d-1}$. Then
\[
N_{i,j}(\rho;i+1,j-2)
=
q^{n+d+i+2j-6}
(q^{i+1}-1)
(q^{n-i-j+1}+q^{d-i-j+1}-q-1),
\qquad \rho\in\{0,p,1\}.
\]
\end{lemma}

\begin{proof}
Put
$N=n-i-j$ and $D=d-i-j$.
We consider the case $(k,l)=(i+1,j-2)$. By the type criteria in
Subsection~\ref{subsec:type-criteria}, the reduced block matrix has type $(i,j)$ if and only if one of the following three mutually exclusive conditions holds:
\[
\begin{array}{ll}
\textup{(a)}&
\alpha_1\neq0,\qquad
\xi_1\neq0,\qquad
\alpha_1^t\xi_1\neq0,\qquad
u_3\in(pR)^N,\qquad
v_3\in(pR)^D,\\[2mm]
\textup{(b)}&
\xi_1\neq0,\qquad
u_3\notin(pR)^N,\qquad
v_3\in(pR)^D,\\[2mm]
\textup{(c)}&
\alpha_1\neq0,\qquad
u_3\in(pR)^N,\qquad
v_3\notin(pR)^D.
\end{array}
\]
In all three cases, no condition is imposed on $\rho'$. Hence the
count is independent of $\rho\in\{0,p,1\}$.

For case \textup{(a)}, the variables
$\beta_1,\alpha_2,\beta_2,\eta_1,\xi_2,\eta_2$
are unrestricted. By
Lemma~\ref{lem:nonzero-inner-product-total-count}, applied in
dimension $i+1$, the number of choices of the finite-field variables
satisfying
\[
\alpha_1^t\xi_1\neq0
\]
is
\[
q^{2i+4j-6} \cdot q^i(q^{i+1}-1)(q-1) 
=
q^{3i+4j-6}(q^{i+1}-1)(q-1).
\]
For this contribution, we require $u_3\in(pR)^N$ and $v_3\in(pR)^D$,
which gives $q^{N+D}$ choices for $(u_3,v_3)$. Hence the first
contribution is
\[
q^{N+D+3i+4j-6}(q^{i+1}-1)(q-1).
\]

For case \textup{(b)}, the variables
$\alpha_1,\beta_1,\alpha_2,\beta_2,\eta_1,\xi_2,\eta_2$
are unrestricted, and $\xi_1$ is nonzero. Therefore the number of
choices for the finite-field variables is
\[
q^{3i+4j-5}(q^{i+1}-1).
\]
For this contribution, we require $u_3\notin(pR)^N$ and $v_3\in(pR)^D$.
Thus there are $q^{N+D}(q^N-1)$
choices for $(u_3,v_3)$. Hence the second contribution is
\[
q^{N+D+3i+4j-5}(q^{i+1}-1)(q^N-1).
\]

Finally, consider the case \textup{(c)}.
The variables
$\beta_1,\alpha_2,\beta_2,\xi_1,\eta_1,\xi_2,\eta_2$
are unrestricted, and $\alpha_1$ is nonzero. Therefore the number of
choices for the finite-field variables is again
\[
q^{3i+4j-5}(q^{i+1}-1).
\]
For this contribution, we require $u_3\in(pR)^N$ and $v_3\notin(pR)^D$.
Thus there are $q^{N+D}(q^D-1)$
choices for $(u_3,v_3)$. Hence the third contribution is
\[
q^{N+D+3i+4j-5}(q^{i+1}-1)(q^D-1).
\]

Adding the three contributions, we obtain
\[
N_{i,j}(\rho;i+1,j-2)
=
q^{N+D+3i+4j-6}(q^{i+1}-1)
(q^{N+1}+q^{D+1}-q-1)
\]
for all $\rho\in\{0,p,1\}$.
Substituting $N=n-i-j$ and $D=d-i-j$ gives the stated formula.
\end{proof}

\begin{lemma}
\label{lem:N-ij-i-j-2}
Assume that $(i,j-2)\in I_{d-1}$. Then
\[
N_{i,j}(\rho;i,j-2)
=
q^{n+d+2i+2j-6}
(q^{n-i-j+1}-1)
(q^{d-i-j+1}-1),
\qquad \rho\in\{0,p,1\}.
\]
\end{lemma}

\begin{proof}
Put
$N=n-i-j+1$ and $D=d-i-j+1$.
We consider the case $(k,l)=(i,j-2)$. 
By the type criteria in
Subsection~\ref{subsec:type-criteria}, the reduced block matrix has
type $(i,j)$ precisely when
\[
u_3\notin(pR)^N,
\qquad
v_3\notin(pR)^D.
\]
In particular, no condition is imposed on $\rho'$. Hence the count is
independent of $\rho\in\{0,p,1\}$.

The finite-field variables are unrestricted, giving $q^{4i+4j-8}$ choices.
There are $^N(q^N-1)$ and $q^D(q^D-1)$ choices for $u_3$ and $v_3$, respectively.
Multiplying these factors gives
\[
N_{i,j}(\rho;i,j-2)
=
q^{N+D+4i+4j-8}
(q^N-1)(q^D-1)
\]
for all $\rho\in\{0,p,1\}$.
Substituting $N=n-i-j+1$ and $D=d-i-j+1$ gives the stated formula.
\end{proof}

\section{Eigenmatrices of the association scheme}

Throughout this section, we keep the notation of Sections~4 and~5.
In particular, $X=\operatorname{Mat}_{d\times n}(R)$ is regarded as
an abelian group under matrix addition, and $X_{i,j}$ denotes the
type class corresponding to $(i,j)\in I_d$.

Since the relations of $\mathfrak X$ are defined by
\[
(x,y)\in \mathcal R_{i,j}
\quad\Longleftrightarrow\quad
y-x\in X_{i,j},
\]
the adjacency matrix corresponding to $\mathcal R_{i,j}$ coincides
with the adjacency matrix of the Cayley graph
\[
\operatorname{Cay}(X,X_{i,j}),
\]
where the vertex set is $X$ and the connection set is $X_{i,j}$.

Since $X$ is a finite abelian group, its irreducible characters are
one-dimensional and indexed by matrices $z\in X$. Fix a canonical
additive character
\[
\psi:R\longrightarrow \mathbb{C}^\times
\]
which is nontrivial on $pR$. The character of the abelian group $X$
associated with $z$ is given by
\[
\chi_z(x)
=
\psi\left(\operatorname{tr}(zx^t)\right),
\]
where $\operatorname{tr}$ denotes the matrix trace.

The adjacency operator of $\operatorname{Cay}(X,X_{i,j})$ acts
diagonally in the Fourier basis, and the eigenvalue corresponding to
$\chi_z$ is
\[
\sum_{x\in X_{i,j}}\chi_z(x)
=
\sum_{x\in X_{i,j}}
\psi\left(\operatorname{tr}(zx^t)\right).
\]
Thus the spectrum of $\operatorname{Cay}(X,X_{i,j})$ is determined by
exponential sums over type classes. These sums depend only on the type
of $z$. For $d\geq 1$, and for $z\in X_{s,t}$, we define
\[
\Delta_{i,j}(s,t;d,n)
=
\sum_{x\in X_{i,j}}
\psi\left(\operatorname{tr}(zx^t)\right),
\qquad
(i,j),(s,t)\in I_d.
\]

For the trivial case $d=0$, we set
\[
\Delta_{0,0}(0,0;0,n)=1.
\]
We use the convention that
\[
\Delta_{i,j}(s,t;d,n)=0
\]
whenever $(i,j)\notin I_d$.

\begin{lemma}
\label{lem:Delta-initial}
Assume that $d \geq1$.
For $(i,j)\in I_d$, one has
\[
\Delta_{i,j}(0,0;d,n)
=
\begin{bmatrix} n\\ j \end{bmatrix}_q
\begin{bmatrix} n-j\\ i \end{bmatrix}_q
q^{j(n+d-i-j)}
\prod_{\nu=0}^{i+j-1}(q^d-q^\nu).
\]
\end{lemma}

\begin{proof}
If $z=0$, then $\chi_z(x)=1$ for all $x\in X$. Hence
\[
\Delta_{i,j}(0,0;d,n)=|X_{i,j}|.
\]
The desired formula follows from
Proposition~\ref{prop:type-class-cardinality}.
\end{proof}

\begin{lemma}
\label{lem:del_RR1}
Assume that $d \ge 1$ and $s\geq 1$. For $(i,j), (s,t)\in I_d$, the
following recurrence holds:
\begin{align*}
\Delta_{i,j}(s,t;d,n)
&=
q^{2i+3j}\Delta_{i,j}(s-1,t;d-1,n-1)\\
&\quad+
q^{i+2j-1}
\left(q^n+q^d-q^{i+j}-q^{i+j-1}\right)
\Delta_{i-1,j}(s-1,t;d-1,n-1)\\
&\quad+
q^{j-1}
\left(q^n-q^{i+j-1}\right)
\left(q^d-q^{i+j-1}\right)
\Delta_{i-2,j}(s-1,t;d-1,n-1)\\
&\quad-
q^{n+d+j-2}
\Delta_{i,j-1}(s-1,t;d-1,n-1).
\end{align*}
\end{lemma}

\begin{proof}
We first treat the case $d=1$. Since $s\geq 1$ and
$(s,t)\in I_1$, we have $(s,t)=(1,0)$. Take
$z=(p,0,\ldots,0)\in R^n$.
Then $X=R^n$, and the three type classes are
\[
X_{0,0}=\{0\},\qquad
X_{1,0}=(pR)^n\setminus\{0\},\qquad
X_{0,1}=R^n\setminus(pR)^n.
\]
Moreover,
\[
\psi(\operatorname{tr}(zx^t))=\psi(px_1).
\]
Since $\psi(p\,\cdot)$ is a nontrivial additive character on $R/pR$,
we have
\[
\psi(0) = 1,\qquad
\sum_{x\in (pR)^n}\psi(px_1)=q^n, \qquad
\sum_{x\in R^n}\psi(px_1)=0.
\]
It follows that
\[
\Delta_{0,0}(1,0;1,n)=1,\qquad
\Delta_{1,0}(1,0;1,n)= q^n-1,\qquad
\Delta_{0,1}(1,0;1,n)=-q^n.
\]
On the other hand, for $d-1=0$, the only type is $(0,0)$, and
\[
\Delta_{0,0}(0,0;0,n-1)=1,
\]
while all other $\Delta$-terms vanish by convention. Substituting
$(i,j)=(0,0),(1,0),(0,1)$ into the stated recurrence gives,
respectively,
\[
1,\qquad q^n-1,\qquad -q^n.
\]
Hence the stated recurrence holds for $d=1$. We may therefore assume
$d\geq 2$.

Let
\[
z=
\operatorname{diag}(
\underbrace{p,\ldots,p}_{s},
\underbrace{1,\ldots,1}_{t},
0,\ldots,0)
\in X_{s,t},
\]
and let $z'$ be the submatrix of $z$ obtained by deleting both the
first row and the first column:
\[
z'=
\operatorname{diag}(
\underbrace{p,\ldots,p}_{s-1},
\underbrace{1,\ldots,1}_{t},
0,\ldots,0)
\in \operatorname{Mat}_{(d-1)\times(n-1)}(R).
\]
For each $x\in X_{i,j}$, write
\[
x=
\begin{pmatrix}
\rho & u^t\\
v & y
\end{pmatrix},
\]
where
\[
\rho\in R,\qquad
u\in R^{n-1},\qquad
v\in R^{d-1},\qquad
y\in\operatorname{Mat}_{(d-1)\times(n-1)}(R).
\]
Then
\[
\psi\left(\operatorname{tr}(zx^t)\right)
=
\psi(p\rho)\,
\psi\left(\operatorname{tr}(z' y^t)\right).
\]

By Corollary~\ref{cor:possible-subtypes}, only pairs
$(k,l)\in S_{i,j}$ can occur. Hence
\begin{align*}
\Delta_{i,j}(s,t;d,n)
&=
\sum_{x\in X_{i,j}}
\psi\left(\operatorname{tr}(zx^t)\right)\\
&=
\sum_{\rho\in R}
\sum_{(k,l)\in S_{i,j}}
\sum_{\substack{y\\ \operatorname{type}(y)=(k,l)}}
N_{i,j}(\rho;k,l)\,
\psi(p\rho)\,
\psi\left(\operatorname{tr}(z'y^t)\right)\\
&=
\sum_{\rho\in R}
\sum_{(k,l)\in S_{i,j}}
N_{i,j}(\rho;k,l)\,
\psi(p\rho)\,
\Delta_{k,l}(s-1,t;d-1,n-1).
\end{align*}
Here we used the fact that $N_{i,j}(\rho;k,l)$ depends only on the
type of $\rho$.

We now evaluate the sums over the three types of $\rho$. Since
$p^2=0$, we have $\psi(p\rho)=1$ for all $\rho\in pR$. Thus
\[
\sum_{\rho\in pR\setminus\{0\}}\psi(p\rho)=q-1.
\]
Moreover, by orthogonality of the nontrivial additive character on
$pR$,
\[
\sum_{\rho\in R\setminus pR}\psi(p\rho)=-q.
\]
Therefore
\[
\Delta_{i,j}(s,t;d,n)
=
\sum_{(k,l)\in S_{i,j}}
C_{i,j}(k,l)\,
\Delta_{k,l}(s-1,t;d-1,n-1),
\]
where
\[
C_{i,j}(k,l)
=
N_{i,j}(0;k,l)
+
(q-1)N_{i,j}(p;k,l)
-
qN_{i,j}(1;k,l).
\]

By Lemmas~\ref{lem:N-ij-i+1-j-1},
\ref{lem:N-ij-i-1-j-1},
\ref{lem:N-ij-i+2-j-2},
\ref{lem:N-ij-i+1-j-2}, and
\ref{lem:N-ij-i-j-2}, we have
\[
C_{i,j}(k,l)=0
\]
for
\[
(k,l)\in
\{(i+1,j-1),(i-1,j-1),(i+2,j-2),(i+1,j-2),(i,j-2)\},
\]
because in these cases $N_{i,j}(\rho;k,l)$ is independent of
$\rho$.

It remains to compute $C_{i,j}(k,l)$ for the four remaining pairs.
By Lemmas~\ref{lem:N-ij-ij}, \ref{lem:N-ij-i-1-j},
\ref{lem:N-ij-i-2-j}, and \ref{lem:N-ij-i-j-1}, we obtain
\[
\begin{aligned}
C_{i,j}(i,j)
&=q^{2i+3j},\\
C_{i,j}(i-1,j)
&= q^{i+2j-1}
\left(q^n+q^d-q^{i+j}-q^{i+j-1}\right),\\
C_{i,j}(i-2,j)
&=q^{j-1}
\left(q^n-q^{i+j-1}\right)
\left(q^d-q^{i+j-1}\right),\\
C_{i,j}(i,j-1)
&=
-q^{n+d+j-2}.
\end{aligned}
\]
Substituting these values into the preceding sum gives the desired
recurrence.
\end{proof}

\begin{lemma}
\label{lem:del_RR2}
Assume that $d \ge 1$ and $t\geq 1$. For $(i,j), (s,t)\in I_d$, the
following recurrence holds:
\[
\Delta_{i,j}(s,t;d,n)
=
q^{i+2j}\Delta_{i,j}(s,t-1;d-1,n-1)
-
q^{i+2j-1}\Delta_{i-1,j}(s,t-1;d-1,n-1).
\]
\end{lemma}

\begin{proof}
We first consider the case $d=1$. Since $t\geq 1$ and
$(s,t)\in I_1$, we have $(s,t)=(0,1)$. Take
$z=(1,0,\ldots,0)\in R^n$. Then $X=R^n$, and the three type
classes are
\[
X_{0,0}=\{0\},\qquad
X_{1,0}=(pR)^n\setminus\{0\},\qquad
X_{0,1}=R^n\setminus(pR)^n.
\]
Moreover,
\[
\psi(\operatorname{tr}(zx^t))=\psi(x_1).
\]
Hence
\[
\Delta_{0,0}(0,1;1,n)=1.
\]
Since $\psi$ is nontrivial on $pR$, we have
\[
\sum_{x\in(pR)^n}\psi(x_1)=0,
\]
and therefore
\[
\Delta_{1,0}(0,1;1,n)
=
\sum_{x\in(pR)^n\setminus\{0\}}\psi(x_1)
=
-1.
\]
Also, since $\psi$ is nontrivial on $R$,
\[
\sum_{x\in R^n}\psi(x_1)=0,
\]
and hence
\[
\Delta_{0,1}(0,1;1,n)
=
\sum_{x\in R^n\setminus(pR)^n}\psi(x_1)
=
0.
\]
On the other hand, for $d-1=0$, the only type is $(0,0)$, and
\[
\Delta_{0,0}(0,0;0,n-1)=1,
\]
while all other $\Delta$-terms vanish by convention. Substituting
$(i,j)=(0,0),(1,0),(0,1)$ into the stated recurrence gives,
respectively,
\[
1,\qquad -1,\qquad 0.
\]
Hence the stated recurrence holds for $d=1$. We may therefore assume
$d\geq 2$.

Let
\[
z=
\operatorname{diag}(
1,
\underbrace{p,\ldots,p}_{s},
\underbrace{1,\ldots,1}_{t-1},
0,\ldots,0)
\in X_{s,t},
\]
and let $z'$ be the submatrix of $z$ obtained by deleting both the
first row and the first column:
\[
z'=
\operatorname{diag}(
\underbrace{p,\ldots,p}_{s},
\underbrace{1,\ldots,1}_{t-1},
0,\ldots,0)
\in \operatorname{Mat}_{(d-1)\times(n-1)}(R).
\]
For each $x\in X_{i,j}$, write
\[
x=
\begin{pmatrix}
\rho & u^t\\
v & y
\end{pmatrix},
\]
where
\[
\rho\in R,\qquad
u\in R^{n-1},\qquad
v\in R^{d-1},\qquad
y\in\operatorname{Mat}_{(d-1)\times(n-1)}(R).
\]
Then
\[
\psi\left(\operatorname{tr}(zx^t)\right)
=
\psi(\rho)\,
\psi\left(\operatorname{tr}(z'y^t)\right).
\]
As in the proof of Lemma~\ref{lem:del_RR1}, using
Corollary~\ref{cor:possible-subtypes}, we get
\begin{align*}
\Delta_{i,j}(s,t;d,n)
&=
\sum_{\rho\in R}
\sum_{(k,l)\in S_{i,j}}
N_{i,j}(\rho;k,l)\,
\psi(\rho)\,
\Delta_{k,l}(s,t-1;d-1,n-1).
\end{align*}

Since $\psi$ is nontrivial on $pR$, we have
\[
\sum_{\rho\in pR\setminus\{0\}}\psi(\rho)=-1.
\]
Moreover, since $\psi$ is nontrivial on $R$ and
\[
\sum_{\rho\in pR}\psi(\rho)=0,
\]
we obtain
\[
\sum_{\rho\in R\setminus pR}\psi(\rho)=0.
\]
Thus
\[
\Delta_{i,j}(s,t;d,n)
=
\sum_{(k,l)\in S_{i,j}}
D_{i,j}(k,l)\,
\Delta_{k,l}(s,t-1;d-1,n-1),
\]
where
\[
D_{i,j}(k,l)
=
N_{i,j}(0;k,l)-N_{i,j}(p;k,l).
\]

By Lemmas~\ref{lem:N-ij-i-2-j}, \ref{lem:N-ij-i-j-1},
\ref{lem:N-ij-i+1-j-1}, \ref{lem:N-ij-i-1-j-1},
\ref{lem:N-ij-i+2-j-2}, \ref{lem:N-ij-i+1-j-2}, and
\ref{lem:N-ij-i-j-2}, we have
\[
D_{i,j}(k,l)=0
\]
for all pairs
\[
(k,l)\in S_{i,j}\setminus\{(i,j),(i-1,j)\}.
\]
By Lemmas~\ref{lem:N-ij-ij} and \ref{lem:N-ij-i-1-j}, we have
\[
D_{i,j}(i,j)=q^{i+2j},
\qquad
D_{i,j}(i-1,j)=-q^{i+2j-1}.
\]
Substituting these two nonzero values into the preceding sum gives
\[
\Delta_{i,j}(s,t;d,n)
=
q^{i+2j}\Delta_{i,j}(s,t-1;d-1,n-1)
-
q^{i+2j-1}\Delta_{i-1,j}(s,t-1;d-1,n-1),
\]
as required.
\end{proof}

\begin{proof}[Proof of Theorem~\ref{main}]
It suffices to show that
\[
\Delta_{i,j}(s,t;d,n)
=
K^{\mathrm{biaff}}_{i,j}(s,t;q^d,n;q)
\]
for all $0\le d\le n$ and $(i,j),(s,t)\in I_d$.
In the notation of Section~3, this means that we specialize
\[
(i_1,i_2)=(i,j),\qquad
(j_1,j_2)=(s,t),\qquad
a=q^d.
\]

We prove this identity by induction on $d$. For the induction, we
include the trivial case $d=0$. In this case $I_0=\{(0,0)\}$. By
the convention for the character sums,
\[
\Delta_{0,0}(0,0;0,n)=1.
\]
On the other hand, since the affine $q$-Krawtchouk polynomials are
regularized, the bivariate affine $q$-Krawtchouk polynomial is
defined at $a=q^0=1$. Hence, by Lemma~\ref{lem:biaff-initial}, we have
\[
K^{\mathrm{biaff}}_{0,0}(0,0;1,n;q)=1.
\]
Thus the desired identity holds for $d=0$.

Assume now that $d\geq 1$, and suppose that the identity has already
been proved for $d-1$ and for all admissible values of the second
parameter. Since $d\leq n$, we have $d-1\leq n-1$, so the induction
hypothesis applies to the pair $(d-1,n-1)$.

Fix $(s,t)\in I_d$, and consider the two families
\[
\left\{\Delta_{i,j}(s,t;d,n)\right\}_{(i,j)\in I_d},
\qquad
\left\{
K^{\mathrm{biaff}}_{i,j}(s,t;q^d,n;q)
\right\}_{(i,j)\in I_d}.
\]
We prove that these two families coincide.

If $(s,t)=(0,0)$, then the equality follows from
Lemmas~\ref{lem:Delta-initial} and \ref{lem:biaff-initial}.

Assume now that $s+t\geq 1$. If $s\geq 1$, then
\[
(s-1,t)\in I_{d-1}.
\]
The recurrence relations in Lemmas~\ref{lem:del_RR1} and
\ref{lem:biaff_RR1} express both families at $(s,t;d,n)$ in terms
of the corresponding values at
\[
(s-1,t;d-1,n-1).
\]
Terms with type indices outside $I_{d-1}$ vanish by convention for the character sums.
The corresponding polynomial terms also vanish: this is immediate from the defining convention if the indices are outside $I_{n-1}$, and otherwise follows from Lemma~\ref{lem:biaff-boundary-vanishing}.
Hence only terms with type indices in $I_{d-1}$ remain.
By the induction hypothesis, the remaining corresponding values are equal.
Therefore the desired equality follows in this case.

It remains to consider the case $s=0$. Since $s+t\geq 1$, we have
$t\geq 1$, and hence
\[
(0,t-1)\in I_{d-1}.
\]
The recurrence relations in Lemmas~\ref{lem:del_RR2} and
\ref{lem:biaff_RR2} express both families at $(0,t;d,n)$ in terms
of the corresponding values at
\[
(0,t-1;d-1,n-1).
\]
The same boundary argument applies here: terms with type indices outside $I_{d-1}$ vanish by convention for the character sums, and
the corresponding polynomial terms vanish by the defining convention
or by Lemma~\ref{lem:biaff-boundary-vanishing}. 
The same boundary argument and the induction hypothesis give the desired equality.

The induction hypothesis now gives the desired equality.
\end{proof}

\appendix

\section{Finite-field counts}
\label{app:finite-field-counts}

Throughout this appendix, all vector spaces are over the residue field $\mathbb{F}_q$ introduced in Section~5.
In this appendix, the symbols $k$ and $l$ denote arbitrary nonnegative dimensions and are not tied to any fixed type $(k,l)$ from Section~5.

As in Section~5, zero-dimensional vector spaces are allowed. Thus
$\mathbb{F}_q^0$ is understood to be the zero space, and any condition
requiring a nonzero vector in such a space gives no choices.
As in Subsection~\ref{subsec:enumeration-N}, apparent negative powers in boundary cases are understood after algebraic simplification.

\subsection*{One-block counts}
\begin{lemma}
\label{lem:nonzero-inner-product-count}
Let $k\geq 0$ and fix $\omega\in\mathbb{F}_q^\times$.
Then the number of pairs $(\xi,\eta)\in\mathbb{F}_q^k\times\mathbb{F}_q^k$
satisfying $\xi^t\eta=\omega$
is
\[
q^{k-1}(q^k-1).
\]
\end{lemma}
\begin{proof}
The case $k=0$ is immediate, so assume $k\geq1$.
If $\xi=0$, then $\xi^t\eta=0$ for every $\eta\in\mathbb{F}_q^k$,
and hence there is no solution since $\omega\neq 0$.

Now fix a nonzero vector $\xi\in\mathbb{F}_q^k$. The map
\[
\varphi_\xi:\mathbb{F}_q^k\longrightarrow \mathbb{F}_q,\qquad
\eta\longmapsto \xi^t\eta
\]
is a nonzero linear functional. Hence it is surjective, and its kernel
has dimension $k-1$. Therefore each fiber of $\varphi_\xi$ has cardinality
\[
|\ker \varphi_\xi|=q^{k-1}.
\]
In particular, the number of $\eta\in\mathbb{F}_q^k$ satisfying
$\xi^t\eta=\omega$ is $q^{k-1}$.

Since there are $q^k-1$ choices for nonzero $\xi$, the total number of
pairs $(\xi,\eta)$ is
\[
(q^k-1)q^{k-1}.
\]
\end{proof}

\begin{lemma}\label{lem:zero-inner-product-count}
Let $k\geq 0$.
Then the number of pairs $(\xi,\eta)\in\mathbb{F}_q^k\times\mathbb{F}_q^k$
satisfying $\xi^t\eta=0$
is
\[
q^{k-1}(q^k+q-1).
\]
\end{lemma}
\begin{proof}
The case $k=0$ is immediate, so assume $k\geq1$.
For each $\omega\in\mathbb{F}_q^\times$, by Lemma~\ref{lem:nonzero-inner-product-count}, the number
of pairs $(\xi,\eta)\in\mathbb{F}_q^k\times\mathbb{F}_q^k$ satisfying $\xi^t\eta=\omega$
is
\[
q^{k-1}(q^k-1).
\]
Hence the number of pairs for which $\xi^t\eta$ is nonzero is
\[
\sum_{\omega\in\mathbb{F}_q^\times}
\#\{(\xi,\eta)\in\mathbb{F}_q^k\times\mathbb{F}_q^k:\xi^t\eta=\omega\}
=
(q-1)q^{k-1}(q^k-1).
\]
Since the total number of pairs in
$\mathbb{F}_q^k\times\mathbb{F}_q^k$ is $q^{2k}$, the number of pairs
satisfying $\xi^t\eta=0$ is
\[
q^{2k}-(q-1)q^{k-1}(q^k-1) =
q^{k-1}(q^k+q-1).
\]
\end{proof}

\begin{lemma}
\label{lem:not-equal-nonzero-inner-product-count}
Let $k\geq 0$ and fix $\omega\in\mathbb{F}_q^\times$.
Then the number of pairs $(\xi,\eta)\in\mathbb{F}_q^k\times\mathbb{F}_q^k$
satisfying $\xi^t\eta \neq \omega$
is
\[
q^{k-1}(q^{k+1}-q^k+1).
\]
\end{lemma}
\begin{proof}
The case $k=0$ is immediate, so assume $k\geq1$.
The total number of pairs
$(\xi,\eta)\in\mathbb{F}_q^k\times\mathbb{F}_q^k$ is $q^{2k}$.
On the other hand, by Lemma~\ref{lem:nonzero-inner-product-count},
the number of pairs satisfying $\xi^t\eta=\omega$
is
\[
q^{k-1}(q^k-1).
\]
Therefore the number of pairs satisfying $\xi^t\eta\neq\omega$ is
\[
q^{2k}-q^{k-1}(q^k-1) = q^{k-1}(q^{k+1}-q^k+1).
\]
\end{proof}

\begin{lemma}
\label{lem:nonzero-inner-product-total-count}
Let $k\geq 0$.
Then the number of pairs $(\xi,\eta)\in\mathbb{F}_q^k\times\mathbb{F}_q^k$
satisfying $\xi^t\eta \neq 0$
is
\[
q^{k-1}(q^k-1)(q-1).
\]
\end{lemma}
\begin{proof}
The case $k=0$ is immediate, so assume $k\geq1$.
The total number of pairs
$(\xi,\eta)\in\mathbb{F}_q^k\times\mathbb{F}_q^k$ is $q^{2k}$.
By Lemma~\ref{lem:zero-inner-product-count}, the number of pairs satisfying $\xi^t\eta=0$
is
\[
q^{k-1}(q^k+q-1).
\]
Therefore the number of pairs satisfying $\xi^t\eta\neq 0$ is
\[
q^{2k}-q^{k-1}(q^k+q-1) = q^{k-1}(q^k-1)(q-1).
\]
\end{proof}

\begin{lemma}
\label{lem:nonzero-vectors-zero-inner-product-count}
Let $k\geq 0$.
Then the number of pairs
$(\xi,\eta)\in\mathbb{F}_q^k\times\mathbb{F}_q^k$
satisfying
$\xi\neq 0$, $\eta\neq 0$, $\xi^t\eta=0$
is
\[
(q^k-1)(q^{k-1}-1).
\]
\end{lemma}

\begin{proof}
The case $k=0$ is immediate, so assume $k\geq1$.
By Lemma~\ref{lem:zero-inner-product-count}, the number of pairs
$(\xi,\eta)\in\mathbb{F}_q^k\times\mathbb{F}_q^k$
satisfying
$\xi^t\eta=0$
is
\[
q^{k-1}(q^k+q-1).
\]

Among these pairs, those with $\xi=0$ are obtained by choosing
$\eta\in\mathbb{F}_q^k$ arbitrarily, and hence there are $q^k$ such
pairs. Similarly, there are $q^k$ pairs with $\eta=0$. The pair
$(0,0)$ has been counted twice, so the number of pairs satisfying
$\xi^t\eta=0$
and at least one of $\xi,\eta$ is zero is
$2q^k-1.$
Therefore the number of pairs satisfying
$\xi\neq 0$, $\eta\neq 0$, $\xi^t\eta=0$
is
\[
q^{k-1}(q^k+q-1)-(2q^k-1)
=
(q^k-1)(q^{k-1}-1).
\]
\end{proof}

\subsection*{Two-block counts}

\begin{lemma}
\label{lem:two-block-not-equal-nonzero-inner-product-count}
Let $k,l\geq 0$ and fix $\omega\in\mathbb{F}_q^\times$.
Then the number of quadruples
$(\xi_1,\eta_1,\xi_2,\eta_2)
\in
\mathbb{F}_q^k\times\mathbb{F}_q^k
\times
\mathbb{F}_q^l\times\mathbb{F}_q^l$
satisfying
$\xi_1^t\eta_1+\xi_2^t\eta_2\neq \omega$
is
\[
q^{k+l-1}(q^{k+l+1}-q^{k+l}+1).
\]
\end{lemma}

\begin{proof}
The case $k=l=0$ is immediate, so assume $k+l\geq1$.
Set
\[
\xi=
\begin{pmatrix}
\xi_1\\
\xi_2
\end{pmatrix},
\qquad
\eta=
\begin{pmatrix}
\eta_1\\
\eta_2
\end{pmatrix}
\]
as vectors in $\mathbb{F}_q^{k+l}$. Then
\[
\xi^t\eta=\xi_1^t\eta_1+\xi_2^t\eta_2.
\]
Therefore the desired number is the number of pairs
$(\xi,\eta)\in \mathbb{F}_q^{k+l}\times \mathbb{F}_q^{k+l}$
satisfying $\xi^t\eta\neq\omega$. Applying
Lemma~\ref{lem:not-equal-nonzero-inner-product-count} in dimension $k+l$, we obtain
\[
q^{k+l-1}(q^{k+l+1}-q^{k+l}+1).
\]
\end{proof}

\begin{lemma}
\label{lem:two-block-nonzero-inner-product-total-count}
Let $k,l\geq 0$.
Then the number of quadruples
$(\xi_1,\eta_1,\xi_2,\eta_2)
\in
\mathbb{F}_q^k\times\mathbb{F}_q^k
\times
\mathbb{F}_q^l\times\mathbb{F}_q^l$
satisfying
$\xi_1^t\eta_1+\xi_2^t\eta_2\neq 0$
is
\[
q^{k+l-1}(q^{k+l}-1)(q-1).
\]
\end{lemma}

\begin{proof}
The case $k=l=0$ is immediate, so assume $k+l\geq1$.
Set
\[
\xi=
\begin{pmatrix}
\xi_1\\
\xi_2
\end{pmatrix},
\qquad
\eta=
\begin{pmatrix}
\eta_1\\
\eta_2
\end{pmatrix}
\]
as vectors in $\mathbb{F}_q^{k+l}$. Then
\[
\xi^t\eta=\xi_1^t\eta_1+\xi_2^t\eta_2.
\]
Applying Lemma~\ref{lem:nonzero-inner-product-total-count} in dimension $k+l$ gives
\[
q^{k+l-1}(q^{k+l}-1)(q-1).
\]
\end{proof}

\begin{lemma}
\label{lem:two-block-fixed-sum-eta-one-nonzero-count}
Let $k, l\geq 0$ and fix $\omega\in\mathbb{F}_q$.
Then the number of quadruples
$(\xi_1,\eta_1,\xi_2,\eta_2)
\in
\mathbb{F}_q^k\times\mathbb{F}_q^k
\times
\mathbb{F}_q^l\times\mathbb{F}_q^l$
satisfying
$\xi_1^t\eta_1+\xi_2^t\eta_2=\omega$
and $\eta_1\neq 0$ is
\[
q^{k+2l-1}(q^k-1).
\]
\end{lemma}

\begin{proof}
The case $k=0$ is immediate, so assume $k\geq1$.
Fix $\eta_1\in\mathbb{F}_q^k\setminus\{0\}$ and
$(\xi_2,\eta_2)\in\mathbb{F}_q^l\times\mathbb{F}_q^l$. Then the equation
\[
\xi_1^t\eta_1+\xi_2^t\eta_2=\omega
\]
is equivalent to
\[
\xi_1^t\eta_1=\omega-\xi_2^t\eta_2.
\]
Since $\eta_1\neq 0$, the map
\[
\mathbb{F}_q^k\longrightarrow \mathbb{F}_q,
\qquad
\xi_1\longmapsto \xi_1^t\eta_1
\]
is a nonzero linear functional. Hence each fiber has cardinality
$q^{k-1}$. Therefore, for each fixed choice of $\eta_1,\xi_2,\eta_2$,
there are $q^{k-1}$ choices of $\xi_1$.

There are $q^k-1$ choices for $\eta_1\neq 0$, and $q^{2l}$ choices
for $(\xi_2,\eta_2)$. Hence the desired number is
\[
(q^k-1)q^{2l}q^{k-1}
=
q^{k+2l-1}(q^k-1).
\]
\end{proof}

\begin{lemma}
\label{lem:two-block-fixed-sum-xi-one-nonzero-count}
Let $k,l\geq 0$ and fix $\omega\in\mathbb{F}_q$.
Then the number of quadruples
$(\xi_1,\eta_1,\xi_2,\eta_2)
\in
\mathbb{F}_q^k\times\mathbb{F}_q^k
\times
\mathbb{F}_q^l\times\mathbb{F}_q^l$
satisfying
$\xi_1^t\eta_1+\xi_2^t\eta_2=\omega$
and $\xi_1\neq 0$ is
\[
q^{k+2l-1}(q^k-1).
\]
\end{lemma}

\begin{proof}
The bijection
\[
(\xi_1,\eta_1,\xi_2,\eta_2) \longmapsto (\eta_1,\xi_1,\xi_2,\eta_2)
\]
preserves $\xi_1^t\eta_1+\xi_2^t\eta_2$ and exchanges the conditions $\xi_1\neq0$ and $\eta_1\neq0$.
The result therefore follows from Lemma~\ref{lem:two-block-fixed-sum-eta-one-nonzero-count}.
\end{proof}

\subsection*{Sixtuple counts}

\begin{lemma}
\label{lem:sixtuple-zero-cross-fixed-nonzero-sum-count}
Let $k,l\geq 0$ and fix $\omega\in\mathbb{F}_q^\times$.
Then the number of sixtuples
$(\xi_1,\eta_1,\xi_2,\eta_2,\xi_3,\eta_3)
\in
\mathbb{F}_q^k\times\mathbb{F}_q^k
\times
\mathbb{F}_q^l\times\mathbb{F}_q^l
\times
\mathbb{F}_q^l\times\mathbb{F}_q^l$
satisfying
$\xi_2^t\eta_3 = 0$ and
$\xi_1^t\eta_1+\xi_2^t\eta_2+\xi_3^t\eta_3=\omega$
is
\[
q^{k+2l-1}(q^{k+2l-1}+q^{k+l}-q^{k+l-1}-1).
\]
\end{lemma}
\begin{proof}
The case $k=l=0$ is immediate, so assume $k+l\geq1$.
We first fix a pair $(\xi_2,\eta_3)\in \mathbb{F}_q^l\times\mathbb{F}_q^l$
satisfying $\xi_2^t\eta_3=0$, and count the possible choices of
$(\xi_1,\eta_1,\eta_2,\xi_3)$.

We distinguish two cases. First suppose that $(\xi_2,\eta_3)=(0,0)$.
Then the second condition reduces to
\[
\xi_1^t\eta_1=\omega.
\]
By Lemma~\ref{lem:nonzero-inner-product-count}, the number of choices of
$(\xi_1,\eta_1)\in \mathbb{F}_q^k\times\mathbb{F}_q^k$
is
\[
q^{k-1}(q^k-1).
\]
The vectors $\eta_2$ and $\xi_3$ are arbitrary, giving $q^{2l}$
choices. Hence this case contributes
\[
q^{2l}q^{k-1}(q^k-1).
\]

Next suppose that $(\xi_2,\eta_3)\neq (0,0)$.
For fixed $(\xi_1,\eta_1)$, the second condition is equivalent to
\[
\xi_2^t\eta_2+\xi_3^t\eta_3
=
\omega-\xi_1^t\eta_1.
\]
Since $(\xi_2,\eta_3)\neq(0,0)$, the map
\[
\mathbb{F}_q^l\times\mathbb{F}_q^l \longrightarrow \mathbb{F}_q,
\qquad
(\eta_2,\xi_3)\longmapsto \xi_2^t\eta_2+\xi_3^t\eta_3
\]
is a nonzero linear functional. Therefore each fiber has cardinality
\[
q^{2l-1}.
\]
Thus, for each fixed $(\xi_1,\eta_1)$, there are $q^{2l-1}$ choices
of $(\eta_2,\xi_3)$. Since $(\xi_1,\eta_1)$ is arbitrary, this gives
\[
q^{2k}q^{2l-1}
\]
choices for each pair $(\xi_2,\eta_3)\neq(0,0)$ with
$\xi_2^t\eta_3=0$.

By Lemma~\ref{lem:zero-inner-product-count}, applied with $l$ in place
of $k$, the number of pairs $(\xi_2,\eta_3)\in \mathbb{F}_q^l\times\mathbb{F}_q^l$
satisfying $\xi_2^t\eta_3=0$ is
\[
q^{l-1}(q^l+q-1).
\]
Excluding the pair $(0,0)$, the number of nonzero such pairs is
\[
q^{l-1}(q^l+q-1)-1.
\]
Therefore the total number of sixtuples is
\[
\begin{aligned}
& q^{2l}q^{k-1}(q^k-1)
+
\left(q^{l-1}(q^l+q-1)-1\right)q^{2k}q^{2l-1} \\
&=
q^{k+2l-1}
\left(
q^{k+2l-1}+q^{k+l}-q^{k+l-1}-1
\right).
\end{aligned}
\]
\end{proof}

\begin{lemma}
\label{lem:sixtuple-nonzero-cross-zero-total-count}
Let $k, l\geq0$ and fix $\omega\in\mathbb{F}_q^\times$.
Then the number of sixtuples
$(\xi_1,\eta_1,\xi_2,\eta_2,\xi_3,\eta_3)
\in
\mathbb{F}_q^k\times\mathbb{F}_q^k
\times
\mathbb{F}_q^l\times\mathbb{F}_q^l
\times
\mathbb{F}_q^l\times\mathbb{F}_q^l$
satisfying
$\xi_2^t\eta_3 = \omega$ and
$\xi_1^t\eta_1+\xi_2^t\eta_2+\xi_3^t\eta_3=0$
is
\[
q^{2k+3l-2}(q^l-1).
\]
\end{lemma}
\begin{proof}
The case $l=0$ is immediate, so assume $l\geq1$.
By Lemma~\ref{lem:nonzero-inner-product-count}, applied with $l$ in place
of $k$, the number of pairs $(\xi_2,\eta_3)\in\mathbb{F}_q^l\times\mathbb{F}_q^l$
satisfying
$\xi_2^t\eta_3=\omega$
is
\[
q^{l-1}(q^l-1).
\]

Fix such a pair $(\xi_2,\eta_3)$. Since $\omega\neq 0$, we have
$(\xi_2,\eta_3)\neq(0,0)$. For fixed
$(\xi_1,\eta_1)\in\mathbb{F}_q^k\times\mathbb{F}_q^k,$
the condition
\[
\xi_1^t\eta_1+\xi_2^t\eta_2+\xi_3^t\eta_3=0
\]
is equivalent to
\[
\xi_2^t\eta_2+\xi_3^t\eta_3=-\xi_1^t\eta_1.
\]
The map
\[
\mathbb{F}_q^l\times\mathbb{F}_q^l\longrightarrow \mathbb{F}_q,
\qquad
(\eta_2,\xi_3)\longmapsto \xi_2^t\eta_2+\xi_3^t\eta_3
\]
is a nonzero linear functional. Hence each fiber has cardinality
\[
q^{2l-1}.
\]
Thus, for each fixed $(\xi_1,\eta_1)$, there are $q^{2l-1}$ choices
of $(\eta_2,\xi_3)$. Since there are $q^{2k}$ choices of
$(\xi_1,\eta_1)$, the number of choices for
$(\xi_1,\eta_1,\eta_2,\xi_3)$ is
\[
q^{2k}q^{2l-1}.
\]

Multiplying by the number of choices of $(\xi_2,\eta_3)$, we obtain
\[
q^{l-1}(q^l-1)q^{2k}q^{2l-1}
=
q^{2k+3l-2}(q^l-1).
\]
\end{proof}

\begin{lemma}
\label{lem:sixtuple-zero-constrained-count}
Let $k,l\geq 0$.
Then the number of sixtuples
$(\xi_1,\eta_1,\xi_2,\eta_2,\xi_3,\eta_3)
\in
\mathbb{F}_q^k\times\mathbb{F}_q^k
\times
\mathbb{F}_q^l\times\mathbb{F}_q^l
\times
\mathbb{F}_q^l\times\mathbb{F}_q^l$
satisfying
$\xi_2^t\eta_3 = 0$ and
$\xi_1^t\eta_1+\xi_2^t\eta_2+\xi_3^t\eta_3=0$
is
\[
q^{k+2l-1}(q^{k+2l-1}+q^{k+l}-q^{k+l-1}+q-1).
\]
\end{lemma}
\begin{proof}
The case $k=l=0$ is immediate, so assume $k+l\geq1$.
We first count all sixtuples satisfying only the condition
\[
\xi_2^t\eta_3=0.
\]
By Lemma~\ref{lem:zero-inner-product-count}, applied with $l$ in place of
$k$, the number of pairs
$(\xi_2,\eta_3)\in\mathbb{F}_q^l\times\mathbb{F}_q^l$ satisfying
$\xi_2^t\eta_3=0$ is
\[
q^{l-1}(q^l + q - 1).
\]
The remaining variables $\xi_1,\eta_1,\eta_2,\xi_3$ are arbitrary.
Hence the total number of sixtuples satisfying $\xi_2^t\eta_3=0$ is
\[
q^{2k+2l} \cdot q^{l-1}(q^l + q - 1) = q^{2k+3l-1}(q^l + q - 1).
\]

For each fixed nonzero value
$\omega\in\mathbb{F}_q^\times$, Lemma~\ref{lem:sixtuple-zero-cross-fixed-nonzero-sum-count}
shows that the number of such sixtuples satisfying
\[
\xi_1^t\eta_1+\xi_2^t\eta_2+\xi_3^t\eta_3=\omega
\]
is
\[
q^{k+2l-1}
\left(q^{k+2l-1}+q^{k+l}-q^{k+l-1}-1\right).
\]
Since the possible values of
\[
\xi_1^t\eta_1+\xi_2^t\eta_2+\xi_3^t\eta_3
\]
partition the set of sixtuples satisfying $\xi_2^t\eta_3=0$, the
number with sum equal to $0$ is
\[
\begin{aligned}
& q^{2k+3l-1}(q^l + q - 1) \\
&\quad
-(q-1)q^{k+2l-1}
\left(q^{k+2l-1}+q^{k+l}-q^{k+l-1}-1\right).
\end{aligned}
\]
Simplifying, this becomes
\[
q^{k+2l-1}
\left(q^{k+2l-1}+q^{k+l}-q^{k+l-1}+q-1\right).
\]
\end{proof}

\begin{lemma}
\label{lem:sixtuple-zero-cross-nonzero-sum-count}
Let $k,l\geq 0$.
Then the number of sixtuples
$(\xi_1,\eta_1,\xi_2,\eta_2,\xi_3,\eta_3)
\in
\mathbb{F}_q^k\times\mathbb{F}_q^k
\times
\mathbb{F}_q^l\times\mathbb{F}_q^l
\times
\mathbb{F}_q^l\times\mathbb{F}_q^l$
satisfying
$\xi_2^t\eta_3 = 0$
and
$\xi_1^t\eta_1+\xi_2^t\eta_2+\xi_3^t\eta_3\neq 0$
is
\[
q^{k+2l-1}(q-1)
\left(
q^{k+2l-1}+q^{k+l}-q^{k+l-1}-1
\right).
\]
\end{lemma}

\begin{proof}
The case $k=l=0$ is immediate, so assume $k+l\geq1$.
For each $\omega\in\mathbb{F}_q^\times$, by
Lemma~\ref{lem:sixtuple-zero-cross-fixed-nonzero-sum-count}, the number of
sixtuples satisfying $\xi_2^t\eta_3=0$
and $\xi_1^t\eta_1+\xi_2^t\eta_2+\xi_3^t\eta_3=\omega$
is
\[
q^{k+2l-1}
\left(
q^{k+2l-1}+q^{k+l}-q^{k+l-1}-1
\right).
\]
The condition
\[
\xi_1^t\eta_1+\xi_2^t\eta_2+\xi_3^t\eta_3\neq 0
\]
is the disjoint union of the conditions obtained by letting $\omega\in\mathbb{F}_q^\times$.
Since $\mathbb{F}_q^\times$ has $q-1$ elements, the desired number is
\[
(q-1)q^{k+2l-1}
\left(
q^{k+2l-1}+q^{k+l}-q^{k+l-1}-1
\right).
\]
\end{proof}

\begin{lemma}
\label{lem:sixtuple-zero-cross-not-fixed-nonzero-sum-count}
Let $k,l\geq 0$ and fix $\omega\in\mathbb{F}_q^\times$.
Then the number of sixtuples
$(\xi_1,\eta_1,\xi_2,\eta_2,\xi_3,\eta_3)
\in
\mathbb{F}_q^k\times\mathbb{F}_q^k
\times
\mathbb{F}_q^l\times\mathbb{F}_q^l
\times
\mathbb{F}_q^l\times\mathbb{F}_q^l$
satisfying
$\xi_2^t\eta_3 = 0$
and
$\xi_1^t\eta_1+\xi_2^t\eta_2+\xi_3^t\eta_3\neq \omega$
is
\[
q^{k+2l-1}
\left(
q^{k+2l}-q^{k+2l-1}+q^{k+l+1}
-2q^{k+l}+q^{k+l-1}+1
\right).
\]
\end{lemma}

\begin{proof}
The case $k=l=0$ is immediate, so assume $k+l\geq1$.
We first count the number of sixtuples satisfying only the condition
$\xi_2^t\eta_3=0$.
By Lemma~\ref{lem:zero-inner-product-count}, applied with $l$ in place
of $k$, the number of pairs
$(\xi_2,\eta_3)\in\mathbb{F}_q^l\times\mathbb{F}_q^l$
satisfying
$\xi_2^t\eta_3=0$
is
\[
q^{l-1}(q^l+q-1).
\]
For each such pair $(\xi_2,\eta_3)$, the remaining vectors
$\xi_1,\eta_1,\eta_2,\xi_3$
are arbitrary. Hence there are
$q^{2k+2l}$
choices for them. Therefore the number of sixtuples satisfying
$\xi_2^t\eta_3=0$ is
\[
q^{2k+2l}q^{l-1}(q^l+q-1).
\]

On the other hand, by
Lemma~\ref{lem:sixtuple-zero-cross-fixed-nonzero-sum-count}, the number of
sixtuples satisfying
$\xi_2^t\eta_3=0$
and
$\xi_1^t\eta_1+\xi_2^t\eta_2+\xi_3^t\eta_3=\omega$
is
\[
q^{k+2l-1}
\left(
q^{k+2l-1}+q^{k+l}-q^{k+l-1}-1
\right).
\]
Thus the desired number is
\[
\begin{aligned}
& q^{2k+2l}q^{l-1}(q^l+q-1)
-
q^{k+2l-1}
\left(
q^{k+2l-1}+q^{k+l}-q^{k+l-1}-1
\right) \\
&=
q^{k+2l-1}
\left(
q^{k+2l}-q^{k+2l-1}+q^{k+l+1}
-2q^{k+l}+q^{k+l-1}+1
\right).
\end{aligned}
\]
\end{proof}

\begin{lemma}
\label{lem:sixtuple-nonzero-cross-nonzero-sum-count}
Let $k, l\geq 0$ and fix $\omega\in\mathbb{F}_q^\times$.
Then the number of sixtuples
$(\xi_1,\eta_1,\xi_2,\eta_2,\xi_3,\eta_3)
\in
\mathbb{F}_q^k\times\mathbb{F}_q^k
\times
\mathbb{F}_q^l\times\mathbb{F}_q^l
\times
\mathbb{F}_q^l\times\mathbb{F}_q^l$
satisfying
$\xi_2^t\eta_3 = \omega$
and
$\xi_1^t\eta_1+\xi_2^t\eta_2+\xi_3^t\eta_3\neq 0$
is
\[
q^{2k+3l-2}(q-1)(q^l-1).
\]
\end{lemma}

\begin{proof}
The case $l=0$ is immediate, so assume $l\geq1$.
By Lemma~\ref{lem:nonzero-inner-product-count}, applied with $l$ in
place of $k$, the number of pairs
$(\xi_2,\eta_3)\in \mathbb{F}_q^l\times\mathbb{F}_q^l$
satisfying
$\xi_2^t\eta_3=\omega$
is
\[
q^{l-1}(q^l-1).
\]
For each such pair $(\xi_2,\eta_3)$, the remaining vectors
$\xi_1,\eta_1,\eta_2,\xi_3$
are arbitrary if we impose only the condition $\xi_2^t\eta_3=\omega$.
Hence the number of sixtuples satisfying $\xi_2^t\eta_3=\omega$ is
\[
q^{l-1}(q^l-1)q^{2k+2l}
=
q^{2k+3l-1}(q^l-1).
\]

On the other hand, by
Lemma~\ref{lem:sixtuple-nonzero-cross-zero-total-count}, the number of
sixtuples satisfying
$\xi_2^t\eta_3=\omega$
and
$\xi_1^t\eta_1+\xi_2^t\eta_2+\xi_3^t\eta_3=0$
is
\[
q^{2k+3l-2}(q^l-1).
\]
Therefore the number of sixtuples satisfying
$\xi_2^t\eta_3=\omega$
and
$\xi_1^t\eta_1+\xi_2^t\eta_2+\xi_3^t\eta_3\neq 0$
is
\[
q^{2k+3l-1}(q^l-1)
-
q^{2k+3l-2}(q^l-1)
=
q^{2k+3l-2}(q-1)(q^l-1).
\]
\end{proof}

\subsection*{Counts with a zero condition}
\begin{lemma}
\label{lem:sixtuple-nonzero-sum-xi-one-or-eta-two-zero-count}
Let $k,l\geq 0$.
Then the number of sixtuples
$(\xi_1,\eta_1,\xi_2,\eta_2,\xi_3,\eta_3)
\in
\mathbb{F}_q^k\times\mathbb{F}_q^k
\times
\mathbb{F}_q^k\times\mathbb{F}_q^k
\times
\mathbb{F}_q^l\times\mathbb{F}_q^l$
satisfying
$\xi_1^t\eta_1+\xi_2^t\eta_2+\xi_3^t\eta_3 \neq 0$
and such that at least one of $\xi_1$ and $\eta_2$ is zero is
\[
q^{2k+l-1}(q-1)\left(2q^{k+l}-q^l-1\right).
\]
\end{lemma}

\begin{proof}
The case $k=l=0$ is immediate, so assume $k+l\geq1$.
Let $A$ be the set of such sixtuples satisfying $\xi_1=0$, and let
$B$ be the set of such sixtuples satisfying $\eta_2=0$. We compute
$|A\cup B|$ by inclusion--exclusion.

First suppose that $\xi_1=0$. Then $\eta_1$ is arbitrary, giving
$q^k$ choices, and the nonvanishing condition becomes
\[
\xi_2^t\eta_2+\xi_3^t\eta_3\neq 0.
\]
By Lemma~\ref{lem:two-block-nonzero-inner-product-total-count}, applied
in dimensions $k$ and $l$, the number of choices of
$(\xi_2,\eta_2,\xi_3,\eta_3)
\in
\mathbb{F}_q^k\times\mathbb{F}_q^k
\times
\mathbb{F}_q^l\times\mathbb{F}_q^l$
satisfying this condition is
\[
q^{k+l-1}(q^{k+l}-1)(q-1).
\]
Hence
\[
|A|
=
q^k q^{k+l-1}(q^{k+l}-1)(q-1)
=
q^{2k+l-1}(q^{k+l}-1)(q-1).
\]

Similarly, if $\eta_2=0$, then $\xi_2$ is arbitrary and the condition
becomes
\[
\xi_1^t\eta_1+\xi_3^t\eta_3\neq 0.
\]
Again by Lemma~\ref{lem:two-block-nonzero-inner-product-total-count}, we
obtain
\[
|B|
=
q^{2k+l-1}(q^{k+l}-1)(q-1).
\]

It remains to compute $|A\cap B|$. In this case, $\xi_1=0$ and
$\eta_2=0$. Thus $\eta_1$ and $\xi_2$ are arbitrary, giving
$q^{2k}$ choices, and the nonvanishing condition reduces to
\[
\xi_3^t\eta_3\neq 0.
\]
By Lemma~\ref{lem:nonzero-inner-product-total-count}, applied in dimension $l$, the number of choices of
$(\xi_3,\eta_3)\in\mathbb{F}_q^l\times\mathbb{F}_q^l$ satisfying
$\xi_3^t\eta_3\neq 0$ is
\[
q^{l-1}(q^l-1)(q-1).
\]
Therefore
\[
|A\cap B|
=
q^{2k}q^{l-1}(q^l-1)(q-1)
=
q^{2k+l-1}(q^l-1)(q-1).
\]

By inclusion--exclusion, the desired number is
\[
\begin{aligned}
|A\cup B|
&= |A|+|B|-|A\cap B| \\
&=
2q^{2k+l-1}(q^{k+l}-1)(q-1)
-
q^{2k+l-1}(q^l-1)(q-1) \\
&=
q^{2k+l-1}(q-1)
\left(2q^{k+l}-q^l-1\right).
\end{aligned}
\]
\end{proof}

\begin{lemma}
\label{lem:sixtuple-not-fixed-sum-xi-one-or-eta-two-zero-count}
Let $k,l\geq 0$ and fix $\omega\in\mathbb{F}_q^\times$.
Then the number of sixtuples
$(\xi_1,\eta_1,\xi_2,\eta_2,\xi_3,\eta_3)
\in
\mathbb{F}_q^k\times\mathbb{F}_q^k
\times
\mathbb{F}_q^k\times\mathbb{F}_q^k
\times
\mathbb{F}_q^l\times\mathbb{F}_q^l$
satisfying
$\xi_1^t\eta_1+\xi_2^t\eta_2+\xi_3^t\eta_3 \neq \omega$
and such that at least one of $\xi_1$ and $\eta_2$ is zero is
\[
q^{2k+l-1}
\left(
2q^{k+l+1}-2q^{k+l}-q^{l+1}+q^l+1
\right).
\]
\end{lemma}

\begin{proof}
The case $k=l=0$ is immediate, so assume $k+l\geq1$.
Let $A$ be the set of such sixtuples with $\xi_1=0$, and let $B$ be
the set of such sixtuples with $\eta_2=0$. We compute
$|A\cup B|$ by inclusion--exclusion.

First suppose that $\xi_1=0$. Then $\eta_1$ is arbitrary, giving
$q^k$ choices, and the condition becomes
\[
\xi_2^t\eta_2+\xi_3^t\eta_3\neq \omega.
\]
By Lemma~\ref{lem:two-block-not-equal-nonzero-inner-product-count},
applied in dimensions $k$ and $l$, the number of choices of
$(\xi_2,\eta_2,\xi_3,\eta_3)
\in
\mathbb{F}_q^k\times\mathbb{F}_q^k
\times
\mathbb{F}_q^l\times\mathbb{F}_q^l$
satisfying this condition is
\[
q^{k+l-1}
\left(
q^{k+l+1}-q^{k+l}+1
\right).
\]
Hence
\[
|A|
=
q^k q^{k+l-1}
\left(
q^{k+l+1}-q^{k+l}+1
\right)
=
q^{2k+l-1}
\left(
q^{k+l+1}-q^{k+l}+1
\right).
\]

Similarly, if $\eta_2=0$, then $\xi_2$ is arbitrary, and the condition
becomes
\[
\xi_1^t\eta_1+\xi_3^t\eta_3\neq \omega.
\]
Again by Lemma~\ref{lem:two-block-not-equal-nonzero-inner-product-count},
we obtain
\[
|B|
=
q^{2k+l-1}
\left(
q^{k+l+1}-q^{k+l}+1
\right).
\]

It remains to compute $|A\cap B|$. In this case, $\xi_1=0$ and
$\eta_2=0$. Thus $\eta_1$ and $\xi_2$ are arbitrary, giving
$q^{2k}$ choices, and the condition reduces to
\[
\xi_3^t\eta_3\neq \omega.
\]
By Lemma~\ref{lem:not-equal-nonzero-inner-product-count}, applied in dimension $l$, the number of pairs
$(\xi_3,\eta_3)\in\mathbb{F}_q^l\times\mathbb{F}_q^l$
satisfying $\xi_3^t\eta_3\neq\omega$ is
\[
q^{l-1}
\left(
q^{l+1}-q^l+1
\right).
\]
Therefore
\[
|A\cap B|
=
q^{2k}q^{l-1}
\left(
q^{l+1}-q^l+1
\right)
=
q^{2k+l-1}
\left(
q^{l+1}-q^l+1
\right).
\]

By inclusion--exclusion, the desired number is
\[
\begin{aligned}
|A\cup B|
&= |A|+|B|-|A\cap B| \\
&=
2q^{2k+l-1}
\left(
q^{k+l+1}-q^{k+l}+1
\right)
-
q^{2k+l-1}
\left(
q^{l+1}-q^l+1
\right) \\
&=
q^{2k+l-1}
\left(
2q^{k+l+1}-2q^{k+l}-q^{l+1}+q^l+1
\right).
\end{aligned}
\]
\end{proof}

\section*{Acknowledgements}

The author would like to thank the anonymous referees for their careful
reading and valuable comments, which helped improve the presentation of
the paper. This work was supported by JSPS KAKENHI Grant Number
JP23K12953.

\section*{Data Availability Statement}
This article is purely theoretical and does not involve any datasets or generated data.

\bigskip

\noindent
Yuta Watanabe \\
Department of Mathematics Education \\
Aichi University of Education \\
1 Hirosawa, Igaya-cho, Kariya, Aichi 448-8542, Japan. \\
email: \texttt{ywatanabe@auecc.aichi-edu.ac.jp}
\end{document}